\documentclass[12pt]{amsart}

\usepackage{amsmath}	
\usepackage{amssymb}

\newtheorem{theorem}{Theorem}[section]
\newtheorem{lemma}[theorem]{Lemma}
\newtheorem{proposition}[theorem]{Proposition}
\newtheorem{corollary}[theorem]{Corollary}

\theoremstyle{remark}

\newtheorem*{remark}{Remark}

\newcommand{\bC}{\mathbb{C}}
\newcommand{\bG}{\mathbb{G}}

\newcommand{\bN}{\mathbb{N}}
\newcommand{\bP}{\mathbb{P}}
\newcommand{\bQ}{\mathbb{Q}}
\newcommand{\bR}{\mathbb{R}}
\newcommand{\bZ}{\mathbb{Z}}
\newcommand{\cA}{{\mathcal{A}}}
\newcommand{\cB}{{\mathcal{B}}}
\newcommand{\cC}{{\mathcal{C}}}
\newcommand{\cD}{{\mathcal{D}}}

\newcommand{\cG}{{\mathcal{G}}}
\newcommand{\cL}{{\mathcal{L}}}

\newcommand{\cO}{{\mathcal{O}}}
\newcommand{\cS}{{\mathcal{S}}}
\newcommand{\cU}{{\mathcal{U}}}
\newcommand{\cZ}{{\mathcal{Z}}}

\newcommand{\gp}{\mathfrak{p}}

\newcommand{\tP}{\tilde{P}}

\newcommand{\ualpha}{\underline{\alpha}}

\newcommand{\uP}{\mathbf{P}}
\newcommand{\uQ}{\mathbf{Q}}
\newcommand{\utheta}{\underline{\theta}}

\newcommand{\uX}{\mathbf{X}}

\newcommand{\uZ}{\underline{Z}}

\newcommand{\Cmult}{\bC^*} 

\newcommand{\disp}{\displaystyle}
\renewcommand{\div}{\mathrm{div}}
\newcommand{\et}{\quad\mbox{and}\quad}
\newcommand{\fleche}{\longrightarrow}

\newcommand{\Ga}{\bG_\mathrm{a}}
\newcommand{\Gm}{\bG_\mathrm{m}}

\newcommand{\Qbar}{\overline{\bQ}}

\newcommand{\Res}{\mathrm{Res}}

\DeclareMathOperator{\dist}{dist}

\begin{document}

\baselineskip=15pt 

\title[A small value estimate]
{A small value estimate for $\Ga\times \Gm$}
\author{Damien ROY}
\address{
   D\'epartement de Math\'ematiques\\
   Universit\'e d'Ottawa\\
   585 King Edward\\
   Ottawa, Ontario K1N 6N5, Canada}
\email{droy@uottawa.ca}
\subjclass[2000]{Primary 11J85; Secondary 11J81}
\thanks{Work partially supported by NSERC and CICMA}

\begin{abstract}
A small value estimate is a statement providing necessary conditions for the existence of certain sequences of non-zero polynomials with integer coefficients taking small values at points of an algebraic group.  Such statements are desirable for applications to transcendental number theory to analyze the outcome of the construction of an auxiliary function.  In this paper, we present a result of this type for the product $\Ga\times\Gm$ whose underlying group of complex points is $\bC\times \Cmult$.  It shows that if a certain sequence of non-zero polynomials in $\bZ[X_1,X_2]$ take small values at a point $(\xi,\eta)$ together with their first derivatives with respect to the invariant derivation $\partial/\partial X_1 + X_2 (\partial/\partial X_2)$, then both $\xi$ and $\eta$ are algebraic over $\bQ$.  The precise statement involves growth conditions on the degree and norm of these polynomials as well as on the absolute values of their derivatives.  It improves on a direct application of Philippon's criterion for algebraic independence and compares favorably with constructions coming from Dirichlet's box principle.
\end{abstract}

\maketitle

%
%

\section{Introduction}
\label{sec:intro}

In continuation with \cite{R2008} and \cite{R2010}, the aim of the present paper is to develop new tools for algebraic independence in situations where the traditional combination of a criterion for algebraic independence and of a zero estimate does not apply.  The small value estimates that we are looking for, aim at extracting as much information as possible from the global data of a sequence of auxiliary polynomials taking many small values at points of a finitely generated subgroup of a commutative algebraic group.  An ultimate goal would be to prove the conjectural small value estimates proposed in \cite{R2001} and \cite{R2002} and shown there to be equivalent respectively to the standard conjecture of Schanuel and its elliptic analog.  In \cite{R2010} and \cite{R2008}, we established some small value estimates respectively for the additive group $\bC=\Ga(\bC)$ and the multiplicative group $\Cmult=\Gm(\bC)$.  The present paper deals with the group
\[
 \cG = \bC\times\Cmult = (\Ga\times\Gm)(\bC),
\]
and considers a sequence of auxiliary polynomials in $\bZ[X_1,X_2]$ taking small values at a fixed point $(\xi,\eta)\in\cG$ together with some of their derivatives with respect to the $\cG$-invariant differential operator
\[
  \cD_1 = \frac{\partial}{\partial X_1} + X_2 \frac{\partial}{\partial X_2 }.
\]
Upon defining the norm $\|P\|$ of a polynomial $P$ as the largest absolute value of its coefficients, and upon denoting by $\lfloor x \rfloor$ the integer part of a real number $x$, our main result reads as follows:

\begin{theorem}
\label{intro:thm:main}
Let $(\xi,\eta)\in\cG$, and let $\beta,\tau,\nu\in\bR$ with
\begin{equation}
 \label{intro:thm:main:eq1}
 1\le \tau<2,
 \quad
 \beta >\tau
 \et
 \nu > 2+\beta-\tau + \frac{(\tau-1)(2-\tau)}{\beta+1-\tau}.
\end{equation}
Suppose that, for each sufficiently large positive integer $D$, there exists a non-zero polynomial $P_D\in\bZ[X_1,X_2]$ of degree $\le D$ and norm $\le \exp(D^\beta)$ such that
\begin{equation}
 \label{intro:thm:main:eq2}
 \max_{0\le i<3 \lfloor D^\tau \rfloor} |\cD_1^i P_D(\xi,\eta)|
    \le \exp(-D^\nu).
\end{equation}
Then, we have $\xi,\eta\in\Qbar$ and moreover $\cD_1^i P_D(\xi,\eta)=0$
$(0\le i< 3\lfloor D^\tau \rfloor)$ for each sufficiently large integer $D$.
\end{theorem}

Formally, the statement of the theorem can be simplified by omitting the constraint $\tau<2$
from \eqref{intro:thm:main:eq1}, because it is a consequence of the other conditions.
Indeed, it follows from the main result of Tijdeman in \cite{Ti} or an earlier result
of Mahler \cite[p.~88, Formula (7)]{Mah} that, for given $(\xi,\eta)\in \cG$, $\tau\ge 2$,
$\nu>2$, and for any sufficiently large positive integer $D$, there exists no polynomial
$P_D\in \bC[X_1,X_2]$ of degree $\le D$ and norm $\ge 1$ which satisfies
\eqref{intro:thm:main:eq2}.  In particular, for the same choice of parameters, there
exists no non-zero $P_D\in\bZ[X_1,X_2]$ of degree $\le D$ satisfying
\eqref{intro:thm:main:eq2}.  Nevertheless, keeping the condition $\tau<2$ makes easier
to compare the statement of the theorem with the constructions described below.

First, note that, in the conclusion of the theorem, the vanishing of the derivatives
$\cD_1^iP_D$ at the point $(\xi,\eta)$ follows from the assertion that this point is
algebraic.  Indeed, for each sufficiently large integer $D$, the polynomials $\cD_1^iP_D\in\bZ[X_1,X_2]$ with $0\le i < 3\lfloor D^\tau \rfloor$ have length at
most $\exp(2D^\beta)$.  Since their absolute values at the point $(\xi,\eta)$ are
bounded above by $\exp(-D^\nu)$ and since $\nu>\beta$, Liouville's inequality implies
that they all vanish at that point for $D$ large enough, when $\xi$ and $\eta$ are
algebraic.  Conversely, the hypotheses of the theorem are fulfilled by any algebraic
point $(\xi,\eta) \in \Qbar\times\Qbar^{\,*}$ because, for such a point $(\xi,\eta)$
and any choice of parameters $\beta,\tau\in\bR$ with
\[
 0\le \tau<2 \et \beta>\max\{0,2\tau-2\},
\]
an application of Thue-Siegel's lemma shows that, for any sufficiently large positive
integer $D$, there exists a non-zero polynomial $P\in\bZ[X_1,X_2]$ of degree $\le D$
and norm $\le \exp(D^\beta)$ such that $\cD_1^i P_D(\xi,\eta)=0$ for $0\le i<3\lfloor
D^\tau \rfloor$.

It is also interesting to compare the statement of the theorem to constructions that
can be achieved for an arbitrary point of $\cG$ using Dirichlet box principle.  For
any choice of $(\xi,\eta)\in\cG$ and $\beta, \tau, \nu \in \bR$ with
\[
 0 \le \tau <2, \quad \beta>\max\{1,\tau\} \et \nu < 2+\beta-\tau,
\]
a simple application of that principle shows the existence of a sequence of non-zero
polynomials $(P_D)_{D\ge 1}$ in $\bZ[X_1,X_2]$ with $\deg(P_D)\le D$ and $\|P_D\|
\le \exp(D^\beta)$ satisfying \eqref{intro:thm:main:eq2} for each large enough $D$.
Thus, if it is possible to reduce the lower bound on $\nu$ in
\eqref{intro:thm:main:eq1}, it could not be by more than
\[
 \frac{(\tau-1)(2-\tau)}{\beta+1-\tau} \le (\tau-1)(2-\tau) \le \frac{1}{4}.
\]
On the other hand, compared to the lower bound $\nu > 2+\beta$ required by a direct
application of Philippon's criterion \cite[Theorem 2.11]{Ph}, our condition on
$\nu$ represents a gain of at least $\tau-1/4$.

Although there is room for possibly improving the conditions on $\beta$ and $\nu$ in Theorem \ref{intro:thm:main}, the restriction $\tau\ge 1$ is crucial in order to be able to conclude that $\xi$ and $\eta$ are algebraic over $\bQ$ or even that they are algebraically dependent over $\bQ$.  This follows from a construction of Khintchine adapted by Philippon \cite[Appendix]{Ph} which shows that, for any sequence of positive real numbers $(\psi_D)_{D\ge 1}$, there exist algebraically independent numbers $\xi,\eta\in\bC$ and a sequence of non-zero linear forms $(L_D)_{D\ge 1}$ in $\bZ + \bZ X_1 + \bZ X_2$ satisfying $\|L_D\|\le D$ and $|L_D(\xi,\eta)| \le \psi_D$ for each $D\ge 1$.  Here, we apply this result with $\psi_D=\exp(-2D^{\nu-\tau})$ assuming simply that $0\le \tau<1$, $\beta > \tau$ and $\nu > \tau$.  For the corresponding point $(\xi,\eta)\in\cG$ and the corresponding sequence of linear forms $(L_D)_{D\ge 1}$, we set $P_D=L_D^{\lfloor 4D^\tau \rfloor}$ for each $D\ge 1$.  Then, for each sufficiently large $D$, the polynomial $P_D\in\bZ[X_1,X_2]$ is non-zero, has degree $\le D$, norm $\le \exp(D^\beta)$ and satisfies the condition \eqref{intro:thm:main:eq2}.  However, both $\xi$ and $\eta$ are transcendental over $\bQ$.

The proof of our main result combines techniques for multiplicity and zero estimates introduced by W.~D.~Brownawell and D.~W.~Masser in \cite{BM} and by D.~W.~Masser in \cite{Ma}, together with techniques of elimination theory developed by Yu.~V.~Nesterenko \cite{Ne1, Ne2} and P.~Philippon \cite{Ph}, and formalized in \cite{LR} in joint work with M.~Laurent.  We give below a short outline of that proof.  The complete argument occupies Section \ref{sec:proof}.

Arguing by contradiction, we first replace each polynomial $P_D$ by an appropriate homogeneous polynomial $\tP_D$ of degree $D$, and replace the differential operator $\cD_1$ by a corresponding homogeneous operator $\cD$ on $\bC[\uX] := \bC[X_0,X_1,X_2]$.  For each integer $D$, we also define a convex body $\cC_D$ of the homogeneous part $\bC[\uX]_D$ of $\bC[\uX]$ of degree $D$.  It consists of all polynomials of $\bC[\uX]_D$ whose derivatives with respect $\cD$ are small up to order $T:=\lfloor D^\tau \rfloor$ at the point $(1,\xi,\eta)$, with upper bounds on the norms of these polynomials and on the absolute values of their derivatives chosen so that $\cD^i\tP_D$ belongs to $\cC_D$ for each integer $i$ with $0\le i < 2T$.

In Section \ref{sec:constr}, we provide an estimate for the height of $\bP^2$ with respect to convex bodies of this type.  It relies on two types of result.  The first one is a general lower bound for the multiplicity of the resultant proved in Section \ref{sec:result}.  This result is of independent interest and can be read independently of the rest of the present paper.  It implies in particular that the resultant of $\bC[\uX]$ in degree $D$ vanishes with multiplicity at least $T$ at each triple of polynomials of $\bC[\uX]_D$ whose derivatives with respect to $\cD$ vanish up to order $T$ at $(1,\xi,\eta)$.  The other result is an interpolation estimate proved in Section \ref{sec:basic}.  It provides an upper bound for the smallest norm of a homogeneous polynomial with prescribed first $T$ derivatives with respect to $\cD$ at the point $(1,\xi,\eta)$.

Based on the above, the results of Section \ref{sec:constr} also provide, by  successive intersections in $\bP^2$ and suitable selection of irreducible components, a zero-dimensional algebraic subset $Z$ of $\bP^2$ defined over $\bQ$ and irreducible over $\bQ$ with small height with respect to $\cC_D$, and whose associated prime ideal in $\bQ[\uX]$ contains $\cD^i\tP_D$ for $i=0,1,\dots,2T-1$.  This last information leads to a posteriori estimates for the degree and the standard height of $Z$.

The most delicate part of the proof lies in the final descent argument.  Denote by $D^*$ the largest positive integer less than $D$ for which the ideal of $Z$ does not contain all derivatives $\cD^i\tP_{D^*}$ with $0\le i < 2\lfloor(D^*)^\tau\rfloor$.  Then, the Chow form of $Z$ in degree $D^*$ does not vanish in at least one of these derivatives $P^*$.  Its absolute value at $P^*$, being a positive integer, is then bounded below by $1$. On the other hand, the same number is a constant times the product of the absolute values of $P^*$ at representatives $\ualpha\in\bC^3$ with norm $1$ of the points $\alpha$ of $Z(\bC)$, where the constant depends only on $D^*$, the degree of $Z$ and the standard height of $Z$.  In Section \ref{sec:distance}, the absolute values $|P^*(\ualpha)|$ are estimated from above in terms of the projective distance between $\alpha$ and $(1:\xi:\eta)$ as well as the distance between $\alpha$ and the analytic curve $\{(1: \xi+z: \eta e^z)\,;\,z\in\bC\}$.  Another result of Section \ref{sec:distance} provides an estimate from below for the height of $Z$ with respect to $\cC_D$, in term of these distances.  Putting all together and choosing $D$ large enough leads to the required contradiction.  This last step is an adaptation of the idea behind Philippon's metric B\'ezout's inequality \cite[Prop.~2.5]{Ph}.

To conclude this introduction, we would like to add that the present method is not restricted to dealing with only one point.  Its main limitation lies instead in the fact that, at each degree $D$, the number of conditions imposed on $P_D$ needs to be less than the dimension of the space of polynomials of degree at most $D$. The multiplicity estimate of Section \ref{sec:result} is sufficiently general to deal efficiently with polynomials $P_D \in \bZ[X_1,X_2]$ of degree $\le D$ satisfying conditions of the form
\[
 |\cD_1^i P_D(m_1\xi_1+\cdots+m_s\xi_s,\eta_1^{m_1}\cdots\eta_s^{m_s})| \le \exp(-D^\nu)
 \quad (0\le i\le D^\tau, \ 0\le m_1,\dots,m_s\le D^\sigma)
\]
provided that $\tau+s\sigma<2$.  Under the latter assumption, it is also possible to interpolate the above values of $P_D$ using a generalization of the interpolation result of Section \ref{sec:basic} which, for simplicity, we did not include here.  However, the condition $\tau+s\sigma<2$, which is natural to impose when $s=0$, becomes very restrictive for $s\ge 1$.  In particular, it requires a value of $\nu$ that is larger than the one available from the conjectural small value estimates of \cite{R2001}, in order to insure that the logarithmic height of $\bP^2$ with respect to the appropriate convex body is negative.  As a consequence, the conclusion of Theorem \ref{intro:thm:main} is stronger than what one would expect from these conjectures.  Another but less fundamental difference is that Theorem \ref{intro:thm:main} does not impose separate upper bounds on the degrees of $P_D$ in $X_1$ and $X_2$.  Such refinement would have required to work within the theory of multi-projective elimination initiated by P.~Philippon in \cite{Ph1994} and developed by G.~R\'emond in \cite{Re}.

%
%
\section{Preliminaries}
\label{sec:prelim}

In this section, we introduce most of the notation and results of elimination theory that we will need in the following.  The formalism that we use is a simplified version of that of \cite{LR}.  Throughout this paper, $\bN$ stands for the set of non-negative integers, and $\bN^*$ for the set of positive integers.

We fix an integer $m\ge 1$ and for each subring $A$ of $\bC$, we denote by $A[\uX]$ the polynomial ring $A[X_0,\dots,X_m]$ in $m+1$ variables over $A$.  For each $D\in\bN^*$, we denote by $A[\uX]_D$ the homogeneous part of degree $D$ of $A[\uX]$.  More generally, when $I$ is an homogeneous ideal of $A[\uX]$, we write $I_D$ for its homogeneous part of degree $D$. For each $(m+1)$-tuple $\nu=(\nu_0,\dots,\nu_m)$ in $\bN^{m+1}$, we put $|\nu| = \nu_0+\cdots+\nu_m$ and write $\uX^\nu$ to denote the monomial $X_0^{\nu_0}\cdots X_m^{\nu_m}$.  For each point $\ualpha=(\alpha_0,\dots,\alpha_m)\in \bC^{m+1}$, we define the \emph{norm} of $\ualpha$ by $\|\ualpha\| = \max\{|\alpha_0|, \dots, |\alpha_m|\}$ and denote the corresponding point of $\bP^m(\bC)$ by $(\alpha_0:\cdots:\alpha_m)$.  Similarly, we define the \emph{norm} $\|P\|$ of a polynomial $P\in\bC[\uX]$ as the largest absolute value of its coefficients.

We also denote by $\bP^m_\bQ$ the projective $m$-space over $\bQ$.  By a \emph{subvariety} of $\bP^m_\bQ$ we mean the closed integral subscheme of $\bP^m_\bQ$ determined by a homogeneous prime ideal $\gp$ of $\bQ[\uX]$ distinct from $(X_0,\dots,X_m)$.  Then, the set of points of $Z$ with values in $\bC$, denoted $Z(\bC)$, is the set of common zeros of the elements of $\gp$ in $\bP^m(\bC)$.  For any set $S$ of homogeneous polynomials of $\bQ[\uX]$, we denote by $\cZ(S)$ the closed subscheme of $\bP_\bQ^m$ determined by the ideal of $\bQ[\uX]$ generated by $S$.  Then, $\cZ(S)(\bC)$ is the set of common zeros of the elements of $S$ in $\bP^m(\bC)$.

Let $Z$ be a subvariety of $\bP^m_\bQ$ and let $D$ be a positive integer.  Putting $t=\dim(Z)$, the first section of \cite{Ph} shows the existence of a polynomial map
\[
 F \colon \bC[\uX]_D^{t+1} \fleche \bC
\]
whose zeros are the $(t+1)$-tuples of polynomials $(P_0,\dots,P_t) \in \bC[\uX]_D^{t+1}$ having at least one common zero on $Z(\bC)$, and whose underlying polynomial relative to the basis of $(t+1)$-tuples of monomials $(\uX^{\nu^{(0)}}, \dots, \uX^{\nu^{(t)}})$ with $|\nu^{(0)}|=\cdots=|\nu^{(t)}|=D$ has coefficients in $\bZ$ and is irreducible over $\bZ$.  Such a map $F$ is unique up to multiplication by $\pm 1$. Moreover, $F$ is separately homogeneous of degree $D^t\deg(Z)$ in each of its $t+1$ polynomial arguments.  We call it a \emph{Chow form} of $Z$ in degree $D$.  We define the (logarithmic) \emph{height} $h(Z)$ of $Z$ as the logarithm of the norm of the polynomial underlying its Chow form in degree $1$.  For example, the Chow form of $\bP^m$ in degree $D$ is the \emph{resultant} viewed as a polynomial map $\Res_D \colon \bC[\uX]_D^{m+1} \to \bC$.  As such, it is homogeneous of degree $D^m$ on each factor.  For $D=1$, this resultant is simply the determinant of $m+1$ linear forms in $m+1$ variables.  As a polynomial in the $(m+1)^2$ coefficients of these linear forms, its non-zero coefficients are $\pm 1$ and thus we get $h(\bP^m)=0$.

We define a \emph{convex body} of $\bC[\uX]_D$ to be a compact subset $\cC$ of $\bC[\uX]_D$ with non-empty interior which satisfies $\lambda P + \mu Q \in \cC$ for any $P,Q\in \cC$ and any $\lambda,\mu\in\bC$ with $|\lambda| + |\mu| \le 1$.  Then, for a subvariety $Z$ of $\bP^m_\bQ$ of dimension $t$ and its corresponding Chow form $F$, we define
\[
 h_\cC(Z) = h_\cC(F) = \log \sup\{ |F(P_0,\dots,P_t)| \,;\, P_0,\dots,P_t\in \cC \}.
\]
In the notation of \cite{LR}, this corresponds to the height of $Z$ or $F$ relative to the adelic convex body of $\bQ[\uX]_D$ whose component at infinity is $\cC$ and whose component at each prime number $p$ is the unit ball of $\bC_p[\uX]_D$ with respect to the maximum norm.

Finally, given $t\in\{0,\dots,m\}$, we define a \emph{cycle} of dimension $t$ in $\bP^m_\bQ$ to be a formal linear combination
\[
 Z=m_1Z_1+\cdots+m_sZ_s
\]
with positive integer coefficients $m_1,\dots,m_s$ of subvarieties $Z_1,\dots,Z_s$ of $\bP^m_\bQ$ of dimension $t$.  The latter are called the \emph{components} of $Z$.  We extend the notions of degree and heights to such cycles $Z$ by writing
\[
 \deg(Z)=\sum_{i=1}^s m_i\deg(Z_i),
 \quad
 h(Z)=\sum_{i=1}^s m_ih(Z_i)
 \et
 h_\cC(Z)=\sum_{i=1}^s m_ih_\cC(Z_i),
\]
where $\cC$ stands for an arbitrary convex body of $\bC[\uX]_D$ for some $D\in\bN^*$.

In proving the three results stated below, we freely use the estimate $\dim_\bC\bC[\uX]_D = \binom{D+m}{m} \le (m+1)^D$ valid for any integer $D\ge 0$.

\begin{lemma}
 \label{prelim:lemma:heightB}
Let $D$ be a positive integer and let $\cB=\{P\in\bC[\uX]_D \,;\, \|P\| \le 1\}$.  Then, for any integer $t\in\{0,1,\dots,m\}$ and any cycle $Z$ of $\bP_\bQ^m$ of dimension $t$, we have
\[
 |h_\cB(Z)-D^{t+1}h(Z)| \le (t+4)(t+1)\log(m+1)D^{t+1}\deg(Z).
\]
In particular, we have $h_\cB(\bP^m) \le (m+4)(m+1)\log(m+1)D^{m+1}$.
\end{lemma}

\begin{proof} By linearity, it suffices to prove the first assertion when $Z$ is a subvariety of $\bP^m_\bQ$.  Then, by definition, we have $h(Z)= \log\|F\|$ where $F\colon\bC[\uX]_1^{t+1}\to \bC$ denotes a Chow form of $Z$ in degree $1$.  Similarly, for the convex body $\cD = \{L\in\bC[\uX]_1 \,;\, \|L\| \le 1\}$ of $\bC[\uX]_1$, we have $h_\cD(Z) = h_\cD(F)$ and so, a crude estimate based on \cite[Lemma 3.3 (i)]{LR} gives
\[
 h(Z) \le h_\cD(Z) \le h(Z) + (t+1)\log(m+1)\deg(Z)
\]
because $F$, being homogeneous of degree $\deg(Z)$ on each of the $t+1$ factors of the product $\bC[\uX]_1^{t+1}$, its underlying polynomial has at most $(m+1)^{(t+1)\deg(Z)}$ non-zero coefficients. By Proposition 5.3 of \cite{LR} and the remark stated after it, we also have
\[
 | h_\cB(Z) - D^{t+1}h_\cD(Z) | \le (t+3)(t+1)\log(m+1)D^{t+1}\deg(Z).
\]
Combining the two estimates gives the first assertion.  The second assertion follows from it using $h(\bP^m)=0$.
\end{proof}

\begin{proposition}
 \label{prelim:prop:inter}
Let $D$ be a positive integer, let $\cC$ be a convex body of $\bC[\uX]_D$, and let $Z$ be a subvariety of $\bP_\bQ^m$ of positive dimension $t$.  Suppose that there exists a polynomial $P \in \bZ[\uX]_D \cap \cC$ which does not belong to the ideal of $Z$.  Then there exists a cycle $Z'$ of $\bP^m_\bQ$ of dimension $t-1$ which satisfies
\begin{enumerate}
 \item[(i)] $\deg(Z') = D\deg(Z)$,
 \item[(ii)] $h(Z') \le Dh(Z) + \deg(Z)\log\|P\| + 2(t+5)(t+1)\log(m+1)D\deg(Z)$,
 \item[(iii)] $h_\cC(Z') \le h_\cC(Z) + 2t\log(m+1) D^{t+1}\deg(Z)$.
\end{enumerate}
\end{proposition}

\begin{proof} Define $Z'$ to be the intersection product $Z'=Z\cdot\div(P)$ as in \cite[Lemma 4.2]{LR}.  Then, (i) follows from \cite[Lemma 4.2]{LR} while (iii) derives from \cite[Prop.~4.9]{LR}. To prove (ii), we note that we have $P \in \|P\|\,\cB$ for the convex body $\cB$ of Lemma \ref{prelim:lemma:heightB} and so, by \cite[Prop.~4.9]{LR}, we get
\[
 h_\cB(Z') \le h_{\cB}(Z) + D^t\deg(Z)\log\|P\| + 2t\log(m+1)D^{t+1}\deg(Z).
\]
Then (ii) follows by combining this upper bound with the estimates
\[
 \max\{|h_\cB(Z)-D^{t+1}h(Z)|, |h_\cB(Z')-D^{t}h(Z')|\} \le (t+4)(t+1)\log(m+1)D^{t+1}\deg(Z).
\]
coming from Lemma \ref{prelim:lemma:heightB}.
\end{proof}

\begin{proposition}
 \label{prelim:prop:dim0}
Let $D$ be a positive integer, let $\cC$ be a convex body of $\bC[\uX]_D$, let $Z$ be a subvariety of $\bP_\bQ^m$ of dimension $0$, and let $\uZ$ be a set of representatives of the points of $Z$ by elements of $\bC^{m+1}$ of norm $1$.  Then, we have
\begin{equation}
 \label{prelim:prop:inter:eq1}
 \Big| h_\cC(Z) - Dh(Z) - \sum_{\ualpha\in\uZ} \log\sup\{|P(\ualpha)|\,;\, P\in\cC\} \Big|
 \le 9\log(m+1) D\deg(Z).
\end{equation}
Moreover, if there exists a polynomial $P\in\bZ[\uX]_D \cap \cC$ which does not belong to the ideal of $Z$, then we have $h_\cC(Z) \ge 0$ and
\begin{equation}
 \label{prelim:prop:inter:eq1bis}
 0 \le 7\log(m+1) D \deg(Z) + D h(Z) + \sum_{\ualpha\in\uZ} \log |P(\ualpha)|.
\end{equation}
\end{proposition}

\begin{proof}
Let $F$ be a Chow form of $Z$ in degree $D$.  There is a constant $a\in\Cmult$ depending only on $F$ and $Z$ such that, for any $P\in\bC[\uX]_D$, we have
\begin{equation}
 \label{prelim:prop:inter:fact}
 F(P) = a \prod_{\ualpha\in\uZ} P(\ualpha)\,.
\end{equation}
As this is a factorization of $F$ into a product of $\deg(Z)$ linear forms on $\bC[\uX]_D$ and as $\dim_\bC \bC[\uX]_D \le (m+1)^D$, Proposition 3.7 (i) of \cite{LR} gives
\begin{equation}
 \label{prelim:prop:inter:eq2}
 \Big| h_\cC(Z) - \log|a| -\sum_{\ualpha\in\uZ} \log \sup\{|P(\ualpha)|\,;\, P\in\cC \} \Big|
 \le 2\log(m+1) D \deg(Z).
\end{equation}
Applying this estimate to the convex body $\cB$ of Lemma \ref{prelim:lemma:heightB} instead of $\cC$, we get
\begin{equation}
 \label{prelim:prop:inter:eq3}
 \big| h_\cB(Z) - \log|a| \big|
 \le 3\log(m+1) D \deg(Z)
\end{equation}
because for each of the $\deg(Z)$ points $\ualpha$ of $\uZ$, we have
\[
 0 \le \log \sup\{|P(\ualpha)|\,;\, P\in\cB \} \le \log(m+1) D.
\]
The estimate \eqref{prelim:prop:inter:eq3} combined with Lemma \ref{prelim:lemma:heightB} gives
\begin{equation}
 \label{prelim:prop:inter:eq4}
 \big| \log|a| - D h(Z) \big| \le 7\log(m+1) D \deg(Z),
\end{equation}
which in turn, after substitution into \eqref{prelim:prop:inter:eq2} leads to \eqref{prelim:prop:inter:eq1}.  Finally, if a polynomial $P\in\bZ[\uX]_D \cap \cC$ does not belong to the ideal of $Z$, then we have $F(P)\in\bZ\setminus\{0\}$ and so $h_\cC(Z)\ge \log |F(P)|\ge 0$.  The estimate \eqref{prelim:prop:inter:eq1bis} then follows from the equality \eqref{prelim:prop:inter:fact} together with \eqref{prelim:prop:inter:eq4}.
\end{proof}

%
%
\section{Basic estimates}
\label{sec:basic}

Let $\cG$ denote the commutative group $(\Ga\times\Gm)(\bC)=\bC\times\Cmult$ with its group law written additively.  We denote by $\bC[\uX]$ the ring $\bC[X_0,X_1,X_2]$ and by $\cD$ its homogeneous derivation
\[
 \cD = X_0\frac{\partial}{\partial X_1} + X_2\frac{\partial}{\partial X_2}.
\]
For each $\gamma=(\xi,\eta)\in \cG$ and each $P\in\bC[\uX]$, we define $P(1,\gamma)=P(1,\xi,\eta)$.  We also denote by $\tau_\gamma$ the $\bC$-algebra automorphism of  $\bC[\uX]$ given by
\[
 \tau_\gamma(P(X_0,X_1,X_2))=P(X_0, \xi X_0 + X_1, \eta X_2)
\]
so that, for any $\gamma,\gamma'\in \cG$ and any $P\in\bC[\uX]$, we have
\[
 (\tau_{\gamma}P)(1,\gamma')=P(1,\gamma+\gamma').
\]
We also note that $\tau_{\gamma}\circ\cD = \cD\circ\tau_{\gamma}$ and $\tau_{\gamma}\circ\tau_{\gamma'} = \tau_{\gamma+\gamma'}$ for any $\gamma,\gamma'\in \cG$.  Finally, for each $T\in\bN^*$, we denote by $I^{(\gamma,T)}$ the ideal of $\bC[\uX]$ generated by all homogeneous polynomials $P$ satisfying $\cD^iP(1,\gamma)=0$ for $i=0,\dots,T-1$.  For $L\in\bN$, the symbol $I_L^{(\gamma,T)}$ represents its homogeneous part of degree $L$.

Qualitatively, the results of this section imply that, for fixed $T\in\bN^*$, the largest integer $L$ such that $I_L^{(\gamma,T)}=\{0\}$ is also the largest $L$ with $\binom{L+2}{2} \le T$, and that, for this value of $L$, the ideal $I^{(\gamma,T)}$ is generated by $I_{L+1}^{(\gamma,T)}$ and $I_{L+2}^{(\gamma,T)}$.


We first establish two lemmas where, for a polynomial $Q\in\bC[\uX]$, the notation $\cL(Q)$ stands for the \emph{length} of $Q$, namely the sum of the absolute values of its coefficients.

\begin{lemma}
\label{basic:lemma:derivatives}
Let $D\in\bN$ and $Q\in\bC[\uX]_D$.  For any $\gamma = (\xi,\eta) \in \cG$ and $i\in\bN$, we have
\[
 \cL(\tau_{\gamma}Q) \le c_1(\gamma)^D\|Q\|,
 \quad
 \cL(\cD^iQ)\le D^i\cL(Q)
 \et
 |\cD^iQ(1,\gamma)| \le c_2(\gamma)^D D^i \cL(Q),
\]
where $c_1(\gamma)=2+|\xi|+|\eta|$ and $c_2(\gamma)=\max\{1,|\xi|,|\eta|\}$.
\end{lemma}

\begin{proof}
The first estimate follows from the definition, the second comes from a quick induction on $i$, and the third is a direct consequence of the second.
\end{proof}

\begin{lemma}
\label{basic:lemma:theta}
Let $r_1,\dots,r_s\in\bC$ and $e_0,e_1,\dots,e_s\in\bN^*$.  Put $R=\max\{1,|r_1|,\dots,|r_s|\}$ and $E=e_0+e_1+\cdots+e_s$.  Then, there is a unique polynomial $a(X)\in\bC[X]$ of degree $<e_0$ such that
\begin{equation}
\label{basic:lemma:theta:eq}
 a(X) \prod_{i=1}^s(1-r_iX)^{e_i} \equiv 1 \mod X^{e_0},
\end{equation}
and its length satisfies $\disp \cL(a) \le \binom{E-1}{e_0-1}R^{e_0-1}$.
\end{lemma}

\begin{proof}
Since $1-r_1X,\dots,1-r_sX$ are units in the ring of formal power series $\bC[[X]]$, the congruence \eqref{basic:lemma:theta:eq} is equivalent to
\[
 a(X) \equiv \prod_{i=1}^s (1+r_iX+r_i^2X^2+\cdots)^{e_i} \mod X^{e_0}.
\]
This shows the existence and uniqueness of $a(X)$ and implies that its length is bounded above by the coefficient of $X^{e_0-1}$ in the series
\[
 \Big( \sum_{j=0}^\infty X^j \Big) \prod_{i=1}^s \Big( \sum_{j=0}^\infty |r_i|^j X^j \Big)^{e_i},
\]
itself dominated by $\Big( \sum_{j=0}^\infty R^j X^j \Big)^{E-e_0+1} = \sum_{j=0}^\infty \binom{j+E-e_0}{E-e_0} R^j X^j$.
\end{proof}

A weaker form of the next result, involving a larger constant, can be derived from Malher's formula (7), page 88 of \cite{Mah}.  For the convenience of the reader, we provide an independent, more specific proof based on the fact that, for each polynomial $P\in \bC[\uX]$, the sequence $(\cD^iP(1,0,1))_{i\in\bN}$ is a linear recurrence sequence.

\begin{proposition}
\label{basic:prop:interp}
Let $\gamma=(\xi,\eta)\in \cG$, let $L\in\bN$ and put $M=\binom{L+2}{2}$.  Then the map
\begin{equation}
\label{basic:prop:interp:eq1}
 \begin{aligned}
 \bC[\uX]_L &\fleche \quad \bC^M\\
  Q \ &\longmapsto (\cD^iQ(1,\gamma))_{0\le i< M}
 \end{aligned}
\end{equation}
is an isomorphism of $\bC$-vector spaces.  Moreover, for each $Q\in \bC[\uX]_L$, we have
\begin{equation}
\label{basic:prop:interp:eq1.5}
  \cL(Q) \le c_1(-\gamma)^L 8^M \max_{0\le i< M} |\cD^iQ(1,\gamma)|\,.
\end{equation}
\end{proposition}

\begin{proof}
The second assertion is a quantitative version of the first because it implies that the linear map \eqref{basic:prop:interp:eq1} is injective and so is an isomorphism, its domain and codomain having the same dimension $M$.  Therefore, it suffices to prove the second assertion. To this end, we fix a polynomial $Q\in\bC[\uX]_L$.

We first consider the case where $\gamma=e=(0,1)$ is the neutral element of $\cG$.  We denote by $\cS$ the complex vector space $\bC^\bN$ of sequences of complex numbers indexed by $\bN$, by $\tau\colon\cS\to \cS$ the $\bC$-linear map which sends a sequence $(u_i)_{i\in\bN}$ to the shifted sequence $(u_{i+1})_{i\in\bN}$, and by $\varphi\colon\bC[\uX]_L\to\cS$ the linear map which sends a polynomial $P\in\bC[\uX]_L$ to the sequence $(\cD^i P(1,0,1))_{i\in\bN}$.

For any $j,k\in\bN$ with $j+k\le L$, the monomial $M_{j,k} = X_0^{L-j-k} X_1^j X_2^k$ is mapped by $\varphi$ to the sequence $u^{(j,k)}$ given by
\[
 u^{(j,k)}_i = i(i-1)\cdots(i-j+1) k^{i-j} \quad (i\in \bN)
\]
with the conventions that $i(i-1)\cdots(i-j+1)=0$ when $j>i$ and that $k^{i-j}=\delta_{i,j}$ when $k=0$.  A quick recurrence argument shows that, for any $r\in\bN$, we have
\begin{equation}
\label{basic:prop:interp:eq2}
 (\tau-k)^r u^{(j,k)}
   = \begin{cases}
       j(j-1)\cdots(j-r+1)u^{(j-r,k)} &\text{if $r\le j$,}\\
       0 &\text{if $r>j$,}
     \end{cases}
\end{equation}
and so the initial term of the above sequence \eqref{basic:prop:interp:eq2} is
\begin{equation}
\label{basic:prop:interp:eq3}
 ((\tau-k)^r u^{(j,k)})_0 = r!\,\delta_{j,r}.
\end{equation}

Now, fix a choice of $r,s\in\bN$ with $r+s\le L$.  We use \eqref{basic:prop:interp:eq2} and \eqref{basic:prop:interp:eq3} to construct a linear functional on $\cS$ which maps $u^{(r,s)}$ to $1$ and all other sequences $u^{(j,k)}$ with $j+k\le L$ to $0$.  To this end, we first note that, by Lemma \ref{basic:lemma:theta}, there exists a unique polynomial $a(X)\in\bC[X]$ of degree $\le L-r-s$ such that
\[
 a(X)\prod_{{\scriptstyle k=0,\dots,L}\atop{\scriptstyle k\neq s}}
   \left( 1- \frac{X}{k-s} \right)^{L-k+1}
 \equiv 1 \mod X^{L-r-s+1}
\]
and its length satisfies
\begin{equation}
 \label{basic:prop:interp:eq4}
  \cL(a) \le \binom{M-r-1}{L-r-s} \le 2^M.
\end{equation}
Then the polynomial
\[
 b(X)
  = \frac{1}{r!}(X-s)^r a(X-s)
    \prod_{{\scriptstyle k=0,\dots,L}\atop{\scriptstyle k\neq s}} \left( \frac{X-k}{s-k} \right)^{L-k+1}
\]
is divisible by $(X-k)^{L-k+1}$ for each $k=0,\dots,L$ with $k\neq s$, and satisfies
\[
 b(X)\equiv \frac{1}{r!}(X-s)^r \mod (X-s)^{L-s+1}.
\]
Moreover it has degree $<M$, and length
\begin{equation}
 \label{basic:prop:interp:eq5}
 \begin{aligned}
  \cL(b)
  &\le \frac{(1+s)^r}{r!}(1+s)^{L-r-s} \cL(a)
    \prod_{k\neq s} \left( \frac{1+k}{|s-k|} \right)^{L-k+1}\\
  &\le \frac{(L+1)!\, L! \cdots 2!\, 1!\,}
            {(s!)^{L-s}\,s!\,(s-1)!\cdots 1!\, (L-s)!\, (L-s-1)! \cdots 1!}\, \cL(a)\\
  &= (L+1)!\,\binom{L}{s} \cdots \binom{s+1}{s}\, \cL(a)\\
  &\le 4^M \cL(a),
 \end{aligned}
\end{equation}
using $(L+1)!\le (L+1)^{L+1} \le 2^M$ and $\binom{k}{s} \le 2^k$ for $k=s+1,\dots,L$.  Since the formula \eqref{basic:prop:interp:eq2} implies that each $u^{(j,k)}$ with $j+k\le L$ is annihilated by $(\tau-k)^{L-k+1}$, we deduce that
\[
 b(\tau)u^{(j,k)} =
  \begin{cases}
    0 &\text{if $k\neq s$,}\\
    (1/r!)(\tau-s)^ru^{(j,s)} &\text{if $k=s$,}
  \end{cases}
\]
and so, thanks to \eqref{basic:prop:interp:eq3}, we conclude that
\[
 (b(\tau)u^{(j,k)})_0 = \delta_{j,r}\delta_{k,s}.
\]

As a consequence, if we write the polynomial $Q\in\bC[\uX]_L$ in the standard form
\[
 Q =\sum_{j+k\le L} q_{j,k} X_0^{L-j-k} X_1^j X_2^k = \sum_{j+k\le L} q_{j,k}M_{j,k}
\]
then, in terms of the corresponding sequence $u = \varphi(Q) = \sum_{j+k\le L} q_{j,k}u^{(j,k)}$, we get
\[
 |q_{r,s}| = |(b(\tau) u)_0|
           \le \cL(b) \max_{0\le i< M} |u_i|
           = \cL(b) \max_{0\le i< M} |\cD^iQ(1,0,1)|.
\]
By \eqref{basic:prop:interp:eq4} and \eqref{basic:prop:interp:eq5} we also have
$\cL(b) \le 8^M$.  The choice of $(r,s)$ being arbitrary, we conclude that
\[
 \|Q\| = \max_{r+s\le L} |q_{r,s}| \le 8^M \max_{0\le i< M} |\cD^iQ(1,e)|.
\]

For the general case, we apply the previous result to $\tau_\gamma Q$ instead of $Q$.  Since $\cD^i(\tau_\gamma Q)(1,e) = \tau_\gamma(\cD^i Q)(1,e) = \cD^iQ(1,\gamma)$ for each $i\in\bN$, this gives
\[
 \|\tau_\gamma Q\| \le 8^M \max_{0\le i< M} |\cD^i Q(1,\gamma)|.
\]
The conclusion follows as Lemma \ref{basic:lemma:derivatives} gives  $\cL(Q) \le c_1(-\gamma)^L \|\tau_\gamma Q\|$.
\end{proof}

\begin{corollary}
 \label{basic:cor:degree}
Let $\gamma\in\cG$ and $T\in\bN^*$.  Define $I_\gamma = I^{(\gamma,1)}$.  Then $I_\gamma$ is a prime ideal of rank $2$ and $I^{(\gamma,T)}$ is $I_\gamma$-primary of degree $T$.
\end{corollary}

\begin{proof}
The ideal $I_\gamma$ is generated by the homogeneous polynomials vanishing at the point $(1,\gamma)$.  Therefore it is prime of rank $2$.  As $(I_\gamma)^T \subseteq I^{(\gamma,T)}\subseteq  I_\gamma$, the radical of $I^{(\gamma,T)}$ is $I_\gamma$.  Moreover, for any choice of homogeneous polynomials $P,Q\in\bC[\uX]$ with $P\notin I_\gamma$ and $Q\notin I^{(\gamma,T)}$, we find that $PQ\notin I^{(\gamma,T)}$.  Thus, $I^{(\gamma,T)}$ is $I_\gamma$-primary.  Finally, consider the linear map $\varphi\colon\bC[\uX]\to \bC^T$ given by $\varphi(Q)=(\cD^iQ(1,\gamma))_{0\le i<T}$ for each $Q\in\bC[\uX]$.  Then, $I_D^{(\gamma,T)}$ is the kernel of the restriction of $\varphi$ to $\bC[\uX]_D$, for each $D\in\bN$. Thus, the Hilbert function of $I^{(\gamma,T)}$ is given by $H(I^{(\gamma,T)}\,;\,D) = \dim_\bC\varphi(\bC[\uX]_D)$.  However, Proposition \ref{basic:prop:interp} shows that $\varphi(\bC[\uX]_D)=\bC^T$ when $\binom{D+2}{2}\ge T$.  Thus, for each large enough integer $D$, the value $H(I^{(\gamma,T)}\,;\,D)$ is constant equal to $T$ and so $I^{(\gamma,T)}$ has degree $T$.
\end{proof}

The following lemma provides the inductive step needed in the proof of the next two propositions.

\begin{lemma}
\label{basic:lemma:division}
Let $\gamma = (\xi,\eta) \in \cG$, let $K,L,N,T\in \bN$ with
\begin{equation}
 \label{basic:lemma:division:eq1}
 \binom{L+1}{2} < T \le \binom{L+2}{2}
 \et
 L < K \le N \le \min\left\{2K-L,\frac{3K+2}{2}\right\},
\end{equation}
and let $Q\in I_N^{(\gamma,T)}$.  Then, we can write $Q = \sum_{j=0}^2 X_j^{N-K} Q_j$ for a choice of polynomials $Q_j\in I^{(\gamma,T)}_K$ ($j=0,1,2$) satisfying
\begin{equation}
 \label{basic:lemma:division:eq2}
\sum_{j=0}^2 \cL(Q_j) \le c_3(\gamma)^K (64 K)^T \cL(Q),
\end{equation}
where $c_3(\gamma)=c_1(-\gamma)c_2(\gamma)$.
\end{lemma}

\begin{proof}
Since $N > 3(N-K-1)$, any monomial in $\uX$ of degree $N$ is divisible by at least one of the monomials $X_0^{N-K}$, $X_1^{N-K}$ or $X_2^{N-K}$.  So, we can write
\[
 Q = X_0^{N-K} P_0 + X_1^{N-K} P_1 + X_2^{N-K} P_2
\]
for some homogeneous polynomials $P_0$, $P_1$, $P_2$ of degree $K$ with
\begin{equation}
 \label{basic:lemma:division:eq3}
 \cL(P_0) + \cL(P_1) + \cL(P_2) = \cL(Q).
\end{equation}
Put $M=\binom{L+2}{2}$.  Then, for each $j=1,2$, Proposition \ref{basic:prop:interp} ensures the existence of a unique polynomial $R_j\in\bC[\uX]_L$ satisfying
\[
 \cD^iR_j(1,\gamma)
  = \begin{cases}
      \cD^iP_j(1,\gamma) &\text{for $0\le i <T$,}\\
      0 &\text{for $T\le i< M$,}
    \end{cases}
\]
and shows, with the help of Lemma \ref{basic:lemma:derivatives}, that it has length
\[
 \cL(R_j)
  \le c_1(-\gamma)^L 8^M \max_{0\le i<T}|\cD^iP_j(1,\gamma)|
  \le c_1(-\gamma)^L 8^M c_2(\gamma)^K K^T \cL(P_j).
\]
As $L<K$ and $M\le 2\binom{L+1}{2}+1 \le 2T-1$, the above estimate simplifies to
\begin{equation}
 \label{basic:lemma:division:eq4}
 \cL(R_j)
  \le \frac{1}{8} c_3(\gamma)^K (64 K)^T \cL(P_j) \quad (j=1,2).
\end{equation}
Furthermore, since $2K-L\ge N$, the expressions
\[
 Q_0 := P_0 + X_0^{2K-L-N} ( X_1^{N-K} R_1 + X_2^{N-K} R_2 )
 \et
 Q_j := P_j - X_0^{K-L}R_j
 \quad (j=1,2)
\]
are homogeneous polynomials of degree $K$ which satisfy
\begin{equation}
  \label{basic:lemma:division:eq5}
  X_0^{N-K}Q_0 + X_1^{N-K}Q_1 + X_2^{N-K}Q_2 = Q.
\end{equation}
By construction, we have $Q_1,Q_2\in I^{(\gamma,T)}$.  Since $Q$ as well belongs to $I^{(\gamma,T)}$, we deduce that $X_0^{N-K}Q_0\in I^{(\gamma,T)}$ and so $Q_0\in I^{(\gamma,T)}$ because $X_0\notin I_\gamma$ (see Corollary \ref{basic:cor:degree}).  Thus, \eqref{basic:lemma:division:eq5} provides a decomposition of $Q$ with  polynomials $Q_0,Q_1,Q_2\in I^{(\gamma,T)}_K$.  Using \eqref{basic:lemma:division:eq3} and \eqref{basic:lemma:division:eq4}, we find as announced
\[
 \sum_{j=0}^2 \cL(Q_j)
  \le 2\cL(R_1) + 2\cL(R_2) + \sum_{j=0}^2 \cL(P_j)
  \le c_3(\gamma)^K (64 K)^T \cL(Q).
\]
\end{proof}

On the qualitative side, this lemma has the following useful consequence.

\begin{proposition}
\label{basic:prop:avoidingzeros}
Let $\gamma=(\xi,\eta)\in \cG$, let $D,T\in\bN^*$ with $T\le \binom{D+1}{2}$.  Then, any homogeneous element of $I^{(\gamma,T)}$ of degree $\ge D$ belongs to the ideal $J$ of $\bC[\uX]$ generated by $I^{(\gamma,T)}_D$.  Moreover, for any finite set of points $S$ of $\bP^2(\bC)$ not containing $(1:\xi:\eta)$, there exists an element of $I_D^{(\gamma,T)}$ which does not vanish at any point of $S$.
\end{proposition}

\begin{proof}
The hypotheses on $D$ and $T$ imply that $\binom{L+1}{2} < T \le \binom{L+2}{2}$ for an integer $L$ with $0\le L<D$.  Then, for any $N\in\bN$ with $N\ge D+1$, the conditions \eqref{basic:lemma:division:eq1} of Lemma \ref{basic:lemma:division} are fulfilled with $K=N-1$ and the lemma shows that $I^{(\gamma,T)}_N$ is contained in the ideal of $\bC[\uX]$ generated by $I^{(\gamma,T)}_{N-1}$.  By induction, we conclude that $I^{(\gamma,T)}_N \subseteq J$ for each $N\ge D$.  This proves the first assertion of the proposition.  It also implies that $I^{(\gamma,T)}$ and $J$ have the same zero set in $\bP^2(\bC)$, namely $\{(1:\xi:\eta)\}$, which leads to the second assertion.
\end{proof}

\begin{proposition}
\label{basic:prop:division}
Let $\gamma = (\xi,\eta) \in \cG$ and let $D,T\in \bN^*$ with $3\le D$ and $1\le T \le \binom{[D/3]+1}{2}$.   For any integer $N\ge D$ and any polynomial $Q\in I_N^{(\gamma,T)}$, we can write $Q = \sum_{|\nu|=N-D} \uX^\nu P_\nu$ for a choice of polynomials $P_\nu\in I^{(\gamma,T)}_D$ ($\nu\in\bN^3$, $|\nu|=N-D$) satisfying
\begin{equation}
 \label{basic:prop:division:eq1}
 \sum_{|\nu|=N-D} \cL(P_\nu) \le c_3(\gamma)^{2N} N^{6T\log(N)} \cL(Q).
\end{equation}
\end{proposition}

\begin{proof}
We proceed by induction on $N$.  For $N=D$, the result is clear.  Suppose that $N>D$ and let $Q\in I_N^{(\gamma,T)}$.  We denote by $L$ the unique non-negative integer satisfying $\binom{L+1}{2} < T \le \binom{L+2}{2}$. Then, the hypotheses on $D$ and $T$ imply that $D \ge 3(L+1)$.  We also define $K=D$ if $N\le (3D+2)/2$, and $K = \lfloor 2N/3 \rfloor$ otherwise.  For this choice of $K$, we have $K\ge D$ and the conditions \eqref{basic:lemma:division:eq1} of Lemma \ref{basic:lemma:division} are fulfilled.  Moreover, since $N\ge D+1\ge 4$, we have $64 K \le N^4$ and so this lemma provides polynomials $Q_0,Q_1,Q_2\in I^{(\gamma,T)}_K$ satisfying
\begin{equation}
 \label{basic:prop:division:eq2}
 Q = \sum_{j=0}^2 X_j^{N-K}Q_j
 \et
 \sum_{j=0}^2 \cL(Q_j) \le c_3(\gamma)^K N^{4T} \cL(Q).
\end{equation}
If $K=D$, this decomposition of $Q$ has all the requested properties.  Otherwise, we have $D<K\le 2N/3<N$ and, by induction, we may assume that each $Q_j$ admits a decomposition $Q_j = \sum_{|\nu|=K-D} \uX^\nu P_{j,\nu}$ with polynomials $P_{j,\nu}\in I^{(\gamma,T)}_D$ satisfying
\[
 \sum_{|\nu|=K-D} \cL(P_{j,\nu}) \le c_3(\gamma)^{2K} K^{6T\log(K)}\cL(Q_j).
\]
Substituting these expressions in the decomposition \eqref{basic:prop:division:eq2} of $Q$ and collecting terms, we obtain a new decomposition $Q = \sum_{|\nu|=N-D} \uX^\nu P_\nu$ with polynomials $P_\nu\in I^{(\gamma,T)}_D$ satisfying
\[
 \begin{aligned}
 \sum_{|\nu|=N-D} \cL(P_\nu)
  &\le c_3(\gamma)^{2K} K^{6T\log(K)} \sum_{j=0}^2\cL(Q_j) \\
  &\le c_3(\gamma)^{3K} \exp\big(6T(\log K)^2 + 4T\log N\big) \cL(Q).
 \end{aligned}
\]
As $K\le 2N/3$ and $N\ge 4$, we have $6(\log K)^2 + 4\log N \le 6(\log N)^2$ and so \eqref{basic:prop:division:eq1} holds.
\end{proof}

%
%

\section{Distance}
\label{sec:distance}

Throughout this section, we fix a point $\gamma=(\xi,\eta) \in \cG=\bC\times\Cmult$ and denote by $(1:\gamma)$ the class of $(1,\gamma)$ in $\bP^2(\bC)$.  To alleviate the notation, we simply write $c_1$ and $c_2$ to denote respectively the constants $c_1(\gamma)$ and $c_2(\gamma)$ of Lemma \ref{basic:lemma:derivatives}, and $c_3$ to denote the constant $c_3(\gamma)$ from Lemma \ref{basic:lemma:division}.  In particular, we have
\[
 c_2 = \max\{1,|\xi|,|\eta|\} = \|(1,\gamma)\|.
\]
For each pair of integers $D\ge 0$ and $T\ge 1$, and each point $\alpha\in\bP^2(\bC)$ with representative $\ualpha = (\alpha_0, \alpha_1, \alpha_2) \in \bC^3$ of norm $\|\ualpha\|=1$, we also define
\[
 |I^{(\gamma,T)}_D|_\alpha = \sup\{ |P(\ualpha)| \,;\, P\in I^{(\gamma,T)}_D,\, \|P\|\le 1\},
\]
where $I^{(\gamma, T)}_D$ stands for the homogeneous part of degree $D$ of the ideal $I^{(\gamma, T)}$ introduced in the preceding section.  The goal of this section is to estimate this quantity in terms of the usual projective distance between $\alpha$ and $(1:\gamma)$ defined by
\begin{equation}
\label{distance:dist:def}
 \dist(\alpha,(1:\gamma))
  = \frac{\|\ualpha\wedge(1,\gamma)\|}{\|\ualpha\|\,\|(1,\gamma)\|}
  = c_2^{-1}\max\{|\alpha_1-\alpha_0\xi|,\,|\alpha_2-\alpha_0\eta|,\,|\alpha_1\eta-\alpha_2\xi|\},
\end{equation}
and of the distance from $\alpha$ to the analytic curve $A_\gamma = \{ (1:\xi+z:\eta e^z)\,;\, z\in\bC\}$ defined by
\[
 \dist(\alpha,A_\gamma)
  = \Big|\frac{\alpha_2}{\alpha_0}-\eta\exp\Big(\frac{\alpha_1}{\alpha_0}-\xi\Big)\Big|
\]
when $\alpha_0\neq 0$.  For our first estimate, we use the following lemma.

\begin{lemma}
\label{distance:lemma:dist}
Let $\alpha\in\bP^2(\bC)$ with $\dist(\alpha,(1:\gamma))\le (2c_2)^{-1}$, and let $\ualpha = (\alpha_0, \alpha_1, \alpha_2) \in \bC^3$ be a representative of $\alpha$ with $\|\ualpha\|=1$.  Then we have $|\alpha_0| \ge (2c_2)^{-1}$.
\end{lemma}

\begin{proof}
We have $\|\ualpha-\alpha_0(1,\gamma)\| \le c_2\dist(\alpha,(1:\gamma)) \le 1/2$, so $\|\alpha_0(1,\gamma)\| \ge \|\ualpha\|-1/2 = 1/2$ and therefore $|\alpha_0|\ge (2c_2)^{-1}$.
\end{proof}

\begin{proposition}
\label{distance:prop:devP}
Let $D, T \in \bN^*$, let $P \in \bC[\uX]_D$ with $P\neq 0$, and let $\alpha$, $\ualpha$ be as in Lemma \ref{distance:lemma:dist}.  Then we have
\begin{equation}
 \label{distance:prop:devP:eq1}
 \frac{|P(\ualpha)|}{\|P\|}
 \le c_4 \max_{0\le i<T} \frac{|\cD^iP(1,\xi,\eta)|}{\|P\|}
     + c_4^D \big( \dist(\alpha,(1:\gamma))^T + \dist(\alpha,A_\gamma)\big)
\end{equation}
where $c_4=3c_2\exp(2c_2^2)$.
\end{proposition}

\begin{proof}
Without loss of generality, we may assume that $\|P\|=1$.  We set
\[
 \delta_1 = \frac{\alpha_1}{\alpha_0}-\xi
 \et
 \delta_2 = \frac{\alpha_2}{\alpha_0}-\eta e^{\delta_1},
\]
and consider the entire function $f\colon\bC\to\bC$ given by
\[
 f(z) = P(1,\xi+z, \eta e^z) \quad (z\in\bC).
\]
Since $\alpha_0^D f(\delta_1) = P(\alpha_0,\alpha_1,\alpha_2-\delta_2\alpha_0)$ and since $f^{(i)}(0)=\cD^iP(1,\xi,\eta)$ for each $i\in\bN$, we get
\[
 \begin{aligned}
 |P(\ualpha)|
   &\le |\alpha_0|^D |f(\delta_1)|
        + |P(\alpha_0,\alpha_1,\alpha_2)-P(\alpha_0,\alpha_1,\alpha_2-\delta_2\alpha_0)| \\
   &\le \sum_{i=0}^\infty \frac{1}{i!} |\cD^iP(1,\xi,\eta)|\,|\delta_1|^i
       + (|\alpha_0|+|\alpha_1|+|\alpha_2|+|\delta_2\alpha_0|)^D |\delta_2\alpha_0|\\
   &\le \sum_{i=0}^{T-1} \frac{1}{i!} |\cD^iP(1,\xi,\eta)|\,|\delta_1|^i
       + \sum_{i=T}^\infty \frac{1}{i!} |\cD^iP(1,\xi,\eta)|\,|\delta_1|^i
       + (3+|\delta_2\alpha_0|)^D |\delta_2| \\
   &\le e^{|\delta_1|} \max_{0\le i <T} |\cD^iP(1,\xi,\eta)|
       + \sum_{i=T}^\infty \frac{1}{i!} (3c_2)^D D^i |\delta_1|^i
       + (3+|\delta_2\alpha_0|)^D |\delta_2|
 \end{aligned}
\]
where the last estimate uses the upper bound $|\cD^iP(1,\xi,\eta)| \le c_2^D D^i \cL(P) \le (3c_2)^D D^i$ coming from Lemma \ref{basic:lemma:derivatives}.  To provide an upper bound for the remaining series, we note that, since $|\alpha_0|\ge (2c_2)^{-1}$, we have
\[
 |\delta_1|
  = |\alpha_0|^{-1}|\alpha_1-\alpha_0\xi|
  \le 2 c_2^2 \dist(\alpha,(1:\gamma)).
\]
As, $\dist(\alpha,(1:\gamma))\le (2c_2)^{-1}\le 1$, this gives $|\delta_1| \le c_2$ and, for each integer $i\ge T$, we can write $|\delta_1|^i \le (2c_2^2)^i \dist(\alpha,(1:\gamma))^T$.  Therefore,
\[
 \sum_{i=T}^\infty \frac{1}{i!} (3c_2)^D D^i |\delta_1|^i
  \le (3c_2)^D \dist(\alpha,(1:\gamma))^T \sum_{i=T}^\infty \frac{1}{i!} (2c_2^2D)^i
  \le c_4^D \dist(\alpha,(1:\gamma))^T.
\]
The conclusion follows because $3+|\delta_2\alpha_0| \le 4+|\eta|e^{|\delta_1|} \le 4+c_2e^{c_2} \le c_4$.
\end{proof}

As an immediate consequence, we get:

\begin{corollary}
\label{distance:cor:maj|I|}
With $\alpha,\ualpha$ as in Lemma \ref{distance:lemma:dist} and $D,T\in\bN^*$, we have
\begin{equation*}
 |I^{(\gamma,T)}_D|_\alpha
 \le c_4^D \big( \dist(\alpha,(1:\gamma))^T + \dist(\alpha,A_\gamma)\big).
\end{equation*}
\end{corollary}

We now turn to the problem of finding a lower bound for $|I^{(\gamma,T)}_D|_\alpha$. To this end, we first note the following consequence of Proposition \ref{basic:prop:division}.

\begin{lemma}
\label{distance:lemma:estQalpha}
Let $D, T \in \bN^*$ with $3 \le D\le T\le \binom{\lfloor D/3 \rfloor + 1}{2}$, let $\alpha\in\bP^2(\bC)$ and let  $\ualpha\in\bC^3$ be a representative of $\alpha$ with $\|\ualpha\|=1$.  Then, for any $Q\in I^{(\gamma,T)}_T$, we have
\begin{equation*}
 |Q(\ualpha)|
 \le
 c_3^{2T} T^{6T\log(T)} \cL(Q)\, |I^{(\gamma,T)}_D|_\alpha.
\end{equation*}
\end{lemma}

\begin{proof}
Fix a polynomial $Q\in I^{(\gamma,T)}_T$ and consider a decomposition of $Q$ as given by Proposition \ref{basic:prop:division} for the choice of $N=T$.  Since $|P(\ualpha)| \le \cL(P)\,|I^{(\gamma,T)}_D|_\alpha$ for any $P\in I^{(\gamma,T)}_D$, we obtain
\[
 |Q(\ualpha)|
 \le
 \sum_{|\nu|=T-D}|P_\nu(\ualpha)|
 \le
 c_3^{2T} T^{6T\log(T)} \cL(Q)\, |I^{(\gamma,T)}_D|_\alpha.
\]
\end{proof}

\begin{proposition}
\label{distance:prop:min|I|}
With the notation and hypotheses of Lemma \ref{distance:lemma:estQalpha}, we have
\begin{equation*}
  \dist(\alpha,(1:\gamma))^T \le c_5^T T^{6T\log(T)} |I^{(\gamma,T)}_D|_\alpha\,.
\end{equation*}
Moreover, if\/ $\dist(\alpha,(1:\gamma))\le (2c_2)^{-1}$, we also have
\begin{equation*}
  \dist(\alpha,A_\gamma) \le c_4 c_5^T T^{6T\log(T)} |I^{(\gamma,T)}_D|_\alpha
\end{equation*}
where $c_5 = (2c_2+3) c_3^2$, and $c_4$ is as in Proposition \ref{distance:prop:devP}.
\end{proposition}

\begin{proof}
The formula \eqref{distance:dist:def} shows that $\dist(\alpha,(1:\gamma))=|M(\ualpha)|$ for some linear form $M\in\bC[\uX]_1$ with $M(1,\gamma)=0$ and $\|M\|\le 1$.  Then, as $M^T\in I_T^{(\gamma,T)}$ and $\cL(M)\le 3$, Lemma \ref{distance:lemma:estQalpha} gives
\begin{equation}
 \label{distance:prop:min|I|:eq1}
 \dist(\alpha,(1:\gamma))^T
  = |M(\ualpha)|^T
  \le 3^T c_3^{2T} T^{6T\log(T)} |I^{(\gamma,T)}_D|_\alpha.
\end{equation}
Now, assume that $\dist(\alpha,(1:\gamma))\le (2c_2)^{-1}$, and write $\ualpha=(\alpha_0,\alpha_1,\alpha_2)$.  As the polynomial
\[
 Q(\uX) = X_0^{T-1}X_2 - \eta \sum_{i=0}^{T-1}\frac{1}{i!}(X_1-\xi X_0)^iX_0^{T-i}
\]
also belongs to $I_T^{(\gamma,T)}$, the same result combined with Lemma \ref{distance:lemma:dist} leads to the estimate
\begin{equation}
 \label{distance:prop:min|I|:eq2}
 \Big| \frac{\alpha_2}{\alpha_0} - \eta \sum_{i=0}^{T-1} \frac{1}{i!} \Big( \frac{\alpha_1}{\alpha_0}-\xi\Big)^i \Big|
  = |\alpha_0|^{-T} |Q(\ualpha)|
  \le c_4 (2c_2)^T c_3^{2T} T^{6T\log(T)} |I^{(\gamma,T)}_D|_\alpha
\end{equation}
using $\cL(Q) \le 1+|\eta|\exp(1+|\xi|) \le c_4$. Arguing as in the proof of Proposition \ref{distance:prop:devP}, we also note that, for each integer $i\ge T$, we have $|\alpha_1/\alpha_0-\xi|^i \le \dist(\alpha,(1:\gamma))^T (2c_2^2)^i$ and therefore
\begin{equation}
 \label{distance:prop:min|I|:eq3}
 \Big| \eta \sum_{i=T}^\infty \frac{1}{i!} \Big( \frac{\alpha_1}{\alpha_0}-\xi\Big)^i \Big|
  \le \dist(\alpha,(1:\gamma))^T c_2 \sum_{i=T}^\infty \frac{(2c_2^2)^i}{i!}
  \le c_4 \dist(\alpha,(1:\gamma))^T.
\end{equation}
Combining \eqref{distance:prop:min|I|:eq1}, \eqref{distance:prop:min|I|:eq2} and \eqref{distance:prop:min|I|:eq3}, we get
\[
 \begin{aligned}
 \dist(\alpha,A_\gamma)
  &\le |\alpha_0|^{-T} |Q(\ualpha)| + c_4 \dist(\alpha,(1:\gamma))^T \\
  &\le c_4 ( (2c_2)^T + 3^T ) c_3^{2T} T^{6T\log(T)} |I^{(\gamma,T)}_D|_\alpha.
 \end{aligned}
\]
\end{proof}

%
%
\section{Multiplicity of the resultant}
\label{sec:result}

In this section, we introduce the last crucial tool that we need for the proof of our main theorem.  It consists in a lower bound for the multiplicity of the resultant of homogeneous polynomials in $m+1$ variables at certain $(m+1)$-tuples of such polynomials.  As in Section \ref{sec:prelim}, we set $\uX=(X_0,\dots,X_m)$ where $m$ is any positive integer.  In the applications, we will restrict to $m=2$.

Recall that a \emph{regular sequence} of $\bC[\uX]$ is a finite sequence of polynomials $(P_0,\dots,P_s)$ with $0\le s\le m$ such that $P_0\neq 0$ and such that the multiplication by $P_j$ is injective in $\bC[\uX] / (P_0,\dots,P_{j-1})$ for $j=1,\dots,s$. When $P_0,\dots,P_s$ are homogeneous, this condition is equivalent to asking that the ideal $(P_0,\dots,P_s)$ has rank $s+1$. We first establish a lemma.

\begin{lemma}
\label{result:lemma:decomp}
Let $P_0,\dots,P_m$ be a regular sequence of $\bC[\uX]$ contained in $\bC[\uX]_D$ for some integer $D\ge 1$, and let $\nu$ be an integer with $\nu\ge (m+1)D-m$.  Then there exist subspaces $E_0,\dots,E_m$ of $\bC[\uX]_{\nu-D}$ with $\dim_\bC(E_m)=D^m$ such that
\[
 \bC[\uX]_\nu = E_0P_0 \oplus\cdots\oplus E_mP_m.
\]
\end{lemma}

\begin{proof}
For $j=0,\dots,m$, put $I_j=(P_0,\dots,P_j)$ and, for each integer $\nu\ge 0$, choose a subspace $E_{j+1}(\nu)$ of $\bC[\uX]_\nu$ such that
\[
 \bC[\uX]_\nu = (I_j)_\nu \oplus E_{j+1}(\nu).
\]
Put also $I_{-1}=(0)$ and $E_0(\nu)=\bC[\uX]_\nu$ so that the above holds for $j=-1$, and extend the definitions to negative integers $\nu$ by putting $\bC[\uX]_\nu = (I_j)_\nu = E_{j+1}(\nu) = \{0\}$ for $j=-1,0,\dots,m$ when $\nu<0$.  Then, for each $\nu\in\bZ$ and each $j=0,\dots,m$, we have an exact sequence
\begin{equation}
\label{result:lemma:decomp:eq1}
 0 \fleche
   \Big(\bC[\uX]/I_{j-1}\Big)_{\nu-D}
   \overset{\times P_j}{\fleche}
   \Big(\bC[\uX]/I_{j-1}\Big)_{\nu}
   \fleche
   \Big(\bC[\uX]/I_j\Big)_{\nu}
   \fleche 0,
\end{equation}
where the first non-trivial map comes from multiplication by $P_j$ in $\bC[\uX]$ while the second is induced by the identity map in $\bC[\uX]$.  As the inclusion of $E_{j+1}(\nu)$ in $\bC[\uX]_\nu$ induces an isomorphism between $E_{j+1}(\nu)$ and $(\bC[\uX]/I_j)_{\nu}$ for each $\nu\in\bZ$ and $j=-1,0,\dots,m$, it follows that
\[
 (I_j)_\nu = E_j(\nu-D)P_j \oplus (I_{j-1})_\nu
 \quad
 (\nu\in\bZ,\ 0\le j \le m).
\]
Since $(I_{-1})_\nu = \{0\}$, combining these decompositions leads to
\begin{equation}
\label{result:lemma:decomp:eq2}
 (I_m)_\nu = \bigoplus_{k=0}^m E_{m-k}(\nu-D)P_{m-k}
\end{equation}
for each $\nu\in\bZ$.  On the other hand, at the level of dimensions, the exactness of the sequence \eqref{result:lemma:decomp:eq1} gives
\begin{equation}
\label{result:lemma:decomp:eq3}
 \dim_\bC E_{j+1}(\nu) = \dim_\bC E_j(\nu) - \dim_\bC E_j(\nu-D)
 \quad
 (\nu\in\bZ,\ 0\le j \le m).
\end{equation}
Since $\dim_\bC E_0(\nu) = (\nu+m)\cdots(\nu+1)/m!$ for each $\nu\ge -m$, we deduce by induction that, for each $j=0,\dots,m$, there exists a polynomial $\chi_j(T)\in\bQ[T]$ of degree $m-j$ with leading coefficient $D^j/(m-j)!$ such that $\dim_\bC E_j(\nu)=\chi_j(\nu)$ for each $\nu\ge jD-m$.  In particular, this gives $\dim_\bC E_m(\nu)=D^m$ for each $\nu\ge mD-m$.  Then \eqref{result:lemma:decomp:eq3} with $j=m$ implies that $E_{m+1}(\nu)=\{0\}$ for each $\nu\ge (m+1)D-m$ and so $(I_m)_\nu=\bC[\uX]_\nu$ for these values of $\nu$.  The conclusion of the lemma then follows from \eqref{result:lemma:decomp:eq2}.
\end{proof}

\begin{theorem}
\label{result:thm:mult}
Let $I$ be an homogeneous ideal of $\bC[\uX]$.  Suppose that, for some integer $D\ge 1$, the set of common zeros of the elements of $I_D$ in $\bP^m(\bC)$ is finite and non-empty.  Then, the generic resultant for homogeneous polynomials of degree $D$ in $m+1$ variables vanishes up to order $\deg(I)$ at each point of $(I_D)^{m+1}$.
\end{theorem}

\begin{proof}
Since the elements of $I_D$ have finitely many common zeros in $\bP^m(\bC)$, the subspace $I_D$ of $\bC[\uX]_D$ contains a regular sequence $P_0,\dots,P_{m-1}$ of length $m$.  Moreover, as the elements of $\bC[\uX]_D$ have no common zeros in $\bP^m(\bC)$, this sequence can be extended to a regular sequence $P_0,\dots,P_{m-1},P_m$ for some $P_m\in\bC[\uX]_D$.  Fix an integer $\nu\ge (m+1)D-m$ large enough so that the Hilbert function of $I$ at $\nu$ is $H(I;\nu)=\deg(I)$, and choose subspaces $E_0,\dots,E_m$ of $\bC[\uX]_{\nu-D}$ as in Lemma \ref{result:lemma:decomp}.  For each $(m+1)$-tuple $\uQ=(Q_0,\dots,Q_m)\in \bC[\uX]_D^{m+1}$, we define a linear map
\[
\begin{aligned}
  \varphi_\uQ\,:\, E_0\times\cdots\times E_m &\longrightarrow \bC[\uX]_\nu\\
                    (A_0,\dots,A_m) &\longmapsto A_0Q_0+\cdots+A_mQ_m\,.
\end{aligned}
\]
Then, by construction, for the choice of $\uP=(P_0,\dots,P_m)$, the map $\varphi_\uP$ is an isomorphism.

Form a basis $\cA$ of $E_0\times\cdots\times E_m$ by concatenating bases of $0\times\cdots\times E_i\times\cdots\times 0$ for $i=0,\dots,m$, so that the last $D^m$ elements of $\cA$ form a basis of $0\times\cdots\times 0\times E_m$.  Since $H(I;\nu)=\deg(I)$, the set $I_\nu$ is a subspace of $\bC[\uX]_\nu$ of codimension $\deg(I)$ and so there is also a basis $\cB$ of $\bC[\uX]_\nu$ whose last elements past the first $\deg(I)$ form a basis of $I_\nu$.  For each $\uQ\in E_0\times\cdots\times E_m$, we denote by $M_\uQ$ the matrix of the linear map $\varphi_\uQ$ with respect to the bases $\cA$ and $\cB$ (its $j$-th column represents the coordinates of the image of the $j$-th element of $\cA$ in the basis $\cB$).  Then, the map
\[
 \begin{aligned}
  \Phi\,:\quad \bC[\uX]_D^{m+1}\quad &\longrightarrow \bC\\
            \uQ=(Q_0,\dots,Q_m) &\longmapsto \det(M_\uQ)
 \end{aligned}
\]
is a multihomogenous polynomial map which is homogeneous of degree $D^m$ in the last component $Q_m$.  For each $\uQ \in I_D^{m+1}$, the first $\deg(I)$ rows of $M_\uQ$ vanish because the image of $\varphi_\uQ$ is contained in $I_\nu$. It follows from this that all partial derivatives of $\Phi$ of order less than $\deg(I)$ vanish at each point of $I_D^{m+1}$.  In particular, $\Phi$ vanishes at each point of $I_D^{m+1}$ and so it is divisible by the resultant in degree $D$: we have
\begin{equation}
 \label{result:thm:mult:eq1}
 \Phi(\uQ) = \Psi(\uQ) \Res_D(\uQ)
\end{equation}
where $\Psi\colon \bC[\uX]_D^{m+1}\to \bC$ is also a polynomial map.  Since the resultant is homogeneous of degree $D^m$ on each factor of $\bC[\uX]_D^{m+1}$ and since $\Phi$ is of the same degree on the last factor, the map $\Psi$ has degree $0$ on that factor.  This means that $\Psi(Q_0,\dots,Q_m)$ is independent of $Q_m$. Since $\Phi(\uP)\neq 0$ and since $P_0,\dots,P_{m-1}\in I_D$, we deduce that the restriction of $\Psi$ to $I_D^{m+1}$ is not the zero map and so the condition $\Psi(\uQ)\neq 0$ defines a non-empty Zariski open subset $\cU$ of $I_D^{m+1}$.  As the map $\Phi$ vanishes to order at least $\deg(I)$ at each point of $\cU$, the factorization \eqref{result:thm:mult:eq1} implies that the resultant vanishes up to order $\deg(I)$ at the same points and therefore, by continuity, vanishes up to order $\deg(I)$ at each point of $I_D^{m+1}$.
\end{proof}

\begin{remark}
Let $I$ and $D$ be as in Theorem \ref{result:thm:mult}.  If $I$ contains the defining ideal of an irreducible subvariety $Z$ of $\bP^m_\bQ$ of dimension $r\ge 0$, then an adaptation of the above argument shows more generally that the Chow form of $Z$ in degree $D$ vanishes up to order $\deg(I)$ at each point of $(I_D)^{r+1}$.
\end{remark}

The applications of Theorem \ref{result:thm:mult} that we give below use the notation of Section \ref{sec:basic}.  In particular, we assume from now on that $m=2$ so that $\uX=(X_0,X_1,X_2)$.  We first prove three lemmas.

\begin{lemma}
\label{result:lemma:derivation:irred}
Let $R$ be an irreducible homogeneous polynomial of $\bQ[\uX]$.  Then $R$ divides $\cD R$ if and only if $R$ is a constant multiple of either $X_0$ or $X_2$.
\end{lemma}

\begin{proof}
Suppose first that $R|\cD R$ and let $D=\deg(R)$.  Since $\cD R$ is also homogeneous of degree $D$, this hypothesis means that $R$ is an eigenvector of the differential operator $\cD$ acting on $\bQ[\uX]_D$.  We observe that, for each $k=0,\dots,D$, the subspace $X_2^k\bQ[X_0,X_1]_{D-k}$ is the kernel of $(\cD-k)^{D-k+1}$, and so the product $\prod_{k=0}^D (\cD-k)^{D-k+1}$ induces the zero operator on $\bQ[\uX]_D$.  Thus the eigenvalues of $\cD$ are the integers $0,1,\dots,D$ and we find that, for each $k=0,1,\dots,D$, the eigenspace for $k$ is generated by the monomial $X_2^kX_0^{D-k}$.  So $R$ is a multiple of such a monomial and, as it is irreducible, we conclude that it has degree $D=1$ and is a multiple of either $X_0$ or
 $X_2$.  The converse is clear.
\end{proof}

\begin{lemma}
\label{result:lemma:gcd}
Let $D$ be a positive integer and let $P\in\bQ[\uX]_D$ with $X_0\nmid P$ and $X_2\nmid P$.  If an irreducible homogeneous polynomial $R\in\bQ[\uX]$ divides $P,\cD{P},\dots,\cD^kP$ for some integer $k\ge 0$, then $R^{k+1}$ divides $P$.  In particular, the polynomials $P,\cD P,\dots,\cD^DP$ have no common irreducible factor in $\bQ[\uX]$.
\end{lemma}

\begin{proof}
Let $R$ be an irreducible factor of $P$ in $\bQ[\uX]$, and write $P=R^eQ$ for some positive integer $e\le D$ and some homogeneous polynomial $Q\in\bQ[\uX]$ not divisible by $R$.  Then, for $i=0,\dots,e-1$, the polynomial $\cD^iP$ is divisible by $R$ while $\cD^eP$ is congruent to $(\cD R)^e Q$ modulo $R$.  However, by Lemma \ref{result:lemma:derivation:irred}, the hypothesis on $P$ implies that $R\nmid \cD R$.  So $e$ is the largest integer for which $R$ divides $P,\cD{P},\dots,\cD^{e-1}P$, and the result follows.
\end{proof}

For the next results, we denote respectively by $\pi_1\colon\cG\to\bC$ and by $\pi_2\colon\cG\to\Cmult$ the projections from $\cG=\bC\times\Cmult$ to its first and second factors.

\begin{lemma}
 \label{result:lemma:stab}
Let $R$ be an irreducible homogeneous polynomial of $\bQ[\uX]$. Then $\tau_\gamma(R)$ is irreducible for any $\gamma\in \cG$.  Moreover, assume that $R$ is not a multiple of either $X_0$ or $X_2$, and denote by $\Gamma_R$ the set of all $\gamma\in \cG$ such that $R$ divides $\tau_\gamma(R)$.  Then either $\pi_1(\Gamma_R)$ is reduced to $\{0\}$ or $\pi_2(\Gamma_R)$ is a cyclic subgroup of $\Cmult$ of order at most $\deg(R)$.
\end{lemma}

\begin{proof}
The first assertion follows simply from the fact that each $\tau_\gamma$ is an automorphism of $\bQ[\uX]$. To prove the second one, we first note that $\Gamma_R$ is a subgroup of $\cG$.  Let $\gamma=(\xi,\eta)$ be an arbitrary element of $\Gamma_R$.  Since $\tau_\gamma(R)$ has the same degree as $R$, we have $\tau_\gamma(R)=\lambda R$ for some $\lambda\in\Cmult$.  Writing $R=\sum_{j=0}^d X_2^jA_j(X_0,X_1)$, this condition translates into $\eta^j A_j(X_0,\xi X_0+X_1) = \lambda A_j(X_0,X_1)$ for each $j=0,\dots,d$.  When $A_j\neq 0$, this relation implies that $\eta^j=\lambda$.  So, if there are at least two indices $j$ with $A_j\neq 0$, then $\eta$ is a root of unity of order at most $d$ and, the choice of $(\xi,\eta)\in\Gamma_R$ being arbitrary, we conclude that $\pi_2(\Gamma_R)$ is a finite thus cyclic subgroup of $\Cmult$ of order at most $d$. Otherwise, assuming that $X_0$ and $X_2$ do not divide $R$, we obtain that $R = A_0(X_0,X_1)$ is of positive degree in $X_1$, and the equality $A_0(X_0,\xi X_0+X_1) = \lambda A_0(X_0,X_1)$ implies that $\lambda=1$ and $\xi=0$.  Thus, in that case, we have $\pi_1(\Gamma_R)=\{0\}$.
\end{proof}

\begin{theorem}
 \label{result:thm:Gamma}
Let $\Sigma$ be a non-empty finite subset of $\cG$ and let $T$ be a positive integer.  Denote by $I$ the ideal of\/ $\bC[\uX]$ generated by the homogeneous polynomials $P$ satisfying
\[
 (\cD^iP)(1,\gamma) = 0 \quad \text{for each \ $\gamma\in\Sigma$ \ and each \ $i=0,\dots,T-1$.}
\]
Suppose that there exist a finite subset $\Sigma_1$ of $\cG$ and an integer $T_1\ge 0$ such that
\begin{equation}
 \label{result:thm:Gamma:eq1}
 D < (T_1+1)\min\{|\pi_1(\Sigma_1)|,|\pi_2(\Sigma_1)|\}
 \et
 (T+T_1)|\Sigma+\Sigma_1| < \binom{D+2}{2}
\end{equation}
where $\Sigma+\Sigma_1 = \{ \gamma+\gamma_1 \,;\,\gamma\in\Sigma,\, \gamma_1\in\Sigma_1 \}$ denotes the sumset of\/ $\Sigma$ and $\Sigma_1$ in $\cG$.  Then, the resultant in degree $D$ vanishes up to order $T\,|\Sigma|$ at each point of $(I_D)^3$.
\end{theorem}

\begin{proof}
We have $I=\cap_{\gamma\in\Sigma}I^{(\gamma,T)}$ where, according to Corollary \ref{basic:cor:degree}, the ideals $I^{(\gamma,T)}$ are primary for distinct prime ideals of rank $2$.  Furthermore they all have the same degree $T$, and so $\deg(I)=T\,|\Sigma|$.

The second condition in \eqref{result:thm:Gamma:eq1} implies the existence of a non-zero polynomial $P \in \bC[\uX]_D$ satisfying
\[
 (\cD^iP)(1,\gamma) = 0 \quad \text{for each \ $\gamma\in\Sigma+\Sigma_1$ \ and each \ $i=0,\dots,T+T_1-1$.}
\]
Fix such a polynomial $P$.  If it is divisible by $X_0$ or by $X_2$, then its quotient by that variable possesses the same vanishing property.  Thus, upon dividing $P$ by a suitable monomial of the form $X_0^kX_2^\ell$ and multiplying the result by $X_1^{k+\ell}$ to restore the degree, we may assume that $P$ is not divisible by $X_0$ nor by $X_2$.  By construction, the polynomials $\tau_\gamma(\cD^iP)$ belong to $I$ for each $\gamma\in\Sigma_1$ and each $i=0,\dots,T_1$. We claim that the latter have no non-constant common factor.  For, suppose they have such a common factor $R$.  Choose it to be homogeneous and irreducible.  As $P$ is not divisible by $X_0$ nor by $X_2$, the same holds for $R$. Define $\Gamma_R$ as in Lemma \ref{result:lemma:stab}, and denote by $\Sigma_2$ a minimal subset of $\Sigma_1$ such that $\Sigma_2+\Gamma_R= \Sigma_1+\Gamma_R$.  For any pair of distinct elements $\gamma,\gamma'$ of $\Sigma_2$, we have $\gamma-\gamma'\notin\Gamma_R$, thus $R$ does not divide $\tau_{\gamma-\gamma'}(R)$, and so the irreducible polynomials $\tau_{-\gamma}(R)$ and $\tau_{-\gamma'}(R)$ are not associated.  Moreover, the choice of $R$ implies that $\tau_{-\gamma}(R)$ divides  $\cD^iP$ for $i=0,1,\dots,T_1$.   By Lemma \ref{result:lemma:gcd}, this means that $P$ is divisible by $\tau_{-\gamma}(R)^{T_1+1}$.  Thus $P$ is divisible by $\prod_{\gamma\in\Sigma_2} \tau_{-\gamma}(R)^{T_1+1}$ and so
\begin{equation}
 \label{result:thm:Gamma:eq2}
 D=\deg(P) \ge (T_1+1)\,|\Sigma_2|\,\deg(R).
\end{equation}
According to Lemma \ref{result:lemma:stab}, either we have $\pi_1(\Gamma_R)=\{0\}$ or $\pi_2(\Gamma_R)$ is cyclic of order at most $\deg(R)$.  In the first case, the equality $\Sigma_2+\Gamma_R= \Sigma_1+\Gamma_R$ implies that $\pi_1(\Sigma_2) = \pi_1(\Sigma_1)$ and from \eqref{result:thm:Gamma:eq2} we deduce that $D\ge (T_1+1)\,|\pi_1(\Sigma_1)|$ against the hypothesis \eqref{result:thm:Gamma:eq1}.  In the second case, it implies that $|\pi_2(\Sigma_2)| \ge |\pi_2(\Sigma_1)|/\deg(R)$ and \eqref{result:thm:Gamma:eq2} leads to $D\ge (T_1+1)\,|\pi_2(\Sigma_1)|$ once again in contradiction with \eqref{result:thm:Gamma:eq1}.

Since the polynomials $\tau_\gamma(\cD^iP)$ with $\gamma\in\Sigma_1$ and $i=0,\dots,T_1$ all belong to $I_D$ and share no common factor, the set of zeros of $I_D$ in $\bP^m(\bC)$ is finite.  As this set contains $\Sigma$, it is also non-empty.  Therefore, by Theorem \ref{result:thm:mult}, the resultant in degree $D$ vanishes up to order $\deg(I) = T\,|\Sigma|$ at each point of $(I_D)^3$.
\end{proof}

In the case where $\Sigma$ consists of just one point $\gamma$, the ideal $I$ of the theorem is simply $I^{(\gamma,T)}$, and for the choice of $\Sigma_1=\{e\}$ and $T_1=D$, the condition \eqref{result:thm:Gamma:eq1} reduces to $T\le \binom{D+1}{2}$.  The conclusion becomes:

\begin{corollary}
 \label{result:cor:1pt}
Let $\gamma\in\cG$ and let $D,T\in\bN^*$ with $T\le \binom{D+1}{2}$.  Then, the resultant in degree $D$ vanishes up to order $T$ at each triple $(P,Q,R)$ of elements of $I^{(\gamma,T)}_D$.
\end{corollary}

%
%

\section{Construction of a subvariety of dimension $0$}
\label{sec:constr}

The first part of the proof of our main theorem consists in constructing, for each sufficiently large integer $D$, a zero-dimensional subvariety $Z$ of $\bP^2_\bQ$ with small height relative to a certain convex body.  In this section, we define a convex body $\cC$ of $\bC[\uX]_D = \bC[X_0,X_1,X_2]_D$ of the appropriate form and provide an estimate for the height of $\bP^2$ relative to $\cC$.  Then, we use this to construct a zero-dimensional subvariety $Z$ with small height $h_\cC(Z)$ assuming the existence of a non-zero homogeneous polynomial $P\in\bZ[\uX]_D$ whose first derivatives with respect to $\cD$ belong to $\cC$.  The rest of the section is devoted to a posteriori estimates for the degree and standard height of $Z$.

\begin{proposition}
\label{constr:prop:convex}
Let $D,T\in\bN^*$ and let $Y,U>0$ with
\begin{equation}
 \label{constr:prop:convex:eq1}
 T \le \binom{D+1}{2}
 \et
 2T\log(c_6) \le Y
\end{equation}
where $c_6=8(2+|\xi|+|\eta|^{-1})$.  Then, for the choice of convex body
\[
 \cC = \{ P\in\bC[\uX]_D \,;\, \|P\|\le e^Y, \, \max_{0\le i<T} |\cD^iP(1,\xi,\eta)| \le e^{-U}\},
\]
we have $h_\cC(\bP^2) \le -TU + 3D^2Y + 21\log(3)D^3$.
\end{proposition}

\begin{proof}
Let $\Res_D\colon\bC[\uX]_D^3\to \bC$ denote the generic resultant of $\bP^2$ in degree $D$.  Using the notation of Lemma \ref{prelim:lemma:heightB}, we have, by that lemma,
\begin{equation}
 \label{constr:prop:convex:eq2}
 h_\cB(\Res_D) = h_\cB(\bP^2) \le 18\log(3)D^3.
\end{equation}
By definition, we also have
\[
 h_\cC(\bP^2) = h_\cC(\Res_D) = \log \sup\{|\Res_D(P_0,P_1,P_2)|\,;\, P_0,P_1,P_2\in\cC\}.
\]
As $\cC$ is compact, there exist $P_0,P_1,P_2\in\cC$ for which $h_\cC(R) = \log |\Res_D(P_0,P_1,P_2)|$.

Let $L$ denote the smallest non-negative integer such that $T \le M:=\binom{L+2}{2}$.  Then, we have $L < M \le 2T$, and the first hypothesis in \eqref{constr:prop:convex:eq1} implies furthermore that $L < D$.  For this choice of $L$ and for each $j=0,1,2$, Proposition \ref{basic:prop:interp} ensures the existence of a unique polynomial $Q_j \in \bC[\uX]_L$ such that
\[
 \cD^iQ_j(1,\xi,\eta)
  = \begin{cases}
     \cD^iP_j(1,\xi,\eta) &\text{for $i=0,\dots,T-1$,}\\
     0 &\text{for $i=T,\dots,M-1$,}
     \end{cases}
\]
and shows that it has norm
\[
 \|Q_j\|
  \le c_1(-\gamma)^L 8^M \max_{0\le i\le T-1} |\cD^iP_j(1,\gamma)|
  \le (8c_1(-\gamma))^{2T} e^{-U}
  \le e^{Y-U}
\]
since $8c_1(-\gamma)=c_6$. By construction, the differences $P_j-X_0^{D-L}Q_j$ are elements of $I^{(\gamma,T)}_D$ and so, according to Corollary \ref{result:cor:1pt}, the polynomial
\[
 f(z) = \Res_D(P_0-(1-z)X_0^{D-L}Q_0, \dots , P_2-(1-z)X_0^{D-L}Q_2) \in \bC[z]
\]
vanishes to order at least $T$ at $z=0$.  Applying the standard Schwarz lemma, this leads to
\[
 \begin{aligned}
 \exp(h_\cC(\Res_D))
  &= |f(1)| \\
  &\le e^{-TU} \sup\{|f(z)| \,;\, |z|=e^U \} \\
  &\le e^{-TU} \sup\{|\Res_D(P_0',P_1',P_2')| \,;\, P_j'\in\bC[\uX]_D \text{ and } \|P_j'\| \le 3e^Y \} \\
  &\le e^{-TU} (3e^Y)^{3D^2} \exp(h_\cB(\Res_D)),
 \end{aligned}
\]
where the last estimate follows from the fact that $\Res_D$ is homogeneous of degree $D^2$ on each of its three arguments.  From this, we conclude using \eqref{constr:prop:convex:eq2}.
\end{proof}

\begin{proposition}
\label{constr:prop:subvar}
Let $D$, $T$, $Y$, $U$ and $\cC$ be as in Proposition \ref{constr:prop:convex}.  Define a real number $C>0$ by the condition $TU = C D^2 Y$ and suppose moreover that
\[
 5<C,
 \quad
 D\le T
 \et
 25\log(3)D \le Y.
\]
Finally, suppose that there exists a non-zero homogeneous polynomial $P\in\bZ[\uX]_D$ not divisible by $X_0$ nor by $X_2$ such that $\cD^i P \in \cC$ for $i=0,\dots,2T-1$.  Then, there exists a subvariety $Z$ of $\cZ(\cD^iP\,;\,0\le i< 2T)$ of dimension $0$ with
\begin{equation}
 \label{constr:prop:subvar:eq}
 h_\cC(Z) \le -C''(Y\deg(Z)+Dh(Z)),
\end{equation}
where $C''=(C-5)/6$.
\end{proposition}

\begin{proof}  Since $\bP^2$ has dimension $2$ with $\deg(\bP^2)=1$ and $h(\bP^2)=0$, and since $P$ is a non-zero element of $\bZ[\uX]_D\cap\cC$, Proposition \ref{prelim:prop:inter} ensures the existence of a non-zero cycle $Z'$ of $\bP_\bQ^2$ of dimension $1$ which satisfies
\[
 \deg(Z') = D,
 \quad
 h(Z') \le Y + 42\log(3)D \le 3Y
\]
and also, thanks to Proposition \ref{constr:prop:convex} and the above estimates,
\[
 \begin{aligned}
 h_\cC(Z')
   &\le h_\cC(\bP^2) + 4\log(3)D^3 \\
   &\le -CD^2Y + 3D^2Y + 25\log(3)D^3 \\
   &\le -(C-4)D^2Y \\
   &\le -C'D( Dh(Z') + 3Y\deg(Z')),
 \end{aligned}
\]
where $C'=(C-4)/6$.  From the last estimate and the additivity of the degree and heights on one-dimensional cycles, we deduce the existence of a component $Z_1$ of $Z'$ with
\[
 h_\cC(Z_1) \le -C'D( Dh(Z_1) + 3Y\deg(Z_1)).
\]
By Lemma \ref{result:lemma:gcd}, the polynomials $P,\cD P,\dots,\cD^DP$ have no common irreducible factor in $\bQ[\uX]$.  Therefore, at least one of them does not belong to the ideal of $Z_1$.  Since it has integral coefficients and since, by hypothesis, it belongs to $\cC$, Proposition \ref{prelim:prop:inter} ensures the existence of a non-zero cycle $Z''$ of $\bP_\bQ^2$ of dimension $0$ with
\[
 \begin{aligned}
 \deg(Z'')
  &= D\deg(Z_1), \\[3pt]
 h(Z'')
  &\le Dh(Z_1) + Y\deg(Z_1) + 24\log(3)D\deg(Z_1) \\
  &\le Dh(Z_1) + 2Y\deg(Z_1), \\[3pt]
 h_\cC(Z'')
   &\le h_\cC(Z_1) + 2\log(3)D^2\deg(Z_1) \\
   &\le -C'D( Dh(Z_1) + 3Y\deg(Z_1)) + (1/2)DY\deg(Z_1) \\
   &\le -C''( Dh(Z'') + Y\deg(Z'') ).
 \end{aligned}
\]
Thus, by linearity, at that least one component $Z$ of $Z''$ satisfies \eqref{constr:prop:subvar:eq}.  Since $C>5$, we have $h_\cC(Z)<0$.  So, by Proposition \ref{prelim:prop:dim0}, the ideal of $Z$ contains $\bZ[\uX]_D\cap \cC$ and so contains $\cD^iP$ for $i=0,\dots,2T-1$.
\end{proof}

\begin{lemma}
\label{constr:lemma:ZerosOutsideG}
Let $D\in\bN^*$ and $\alpha = (\alpha_0:\alpha_1:\alpha_2) \in \bP^2(\bC)$.  Suppose that there exists a non-zero polynomial $P\in\bC[\uX]_D$ not divisible by $X_0$ nor by $X_2$ such that $\cD^iP(\alpha)=0$ for $i=0,\dots,D$.  Then, either we have $\alpha_0\alpha_2\neq 0$ or $\alpha$ is one of the points $(0:1:0)$ or $(0:0:1)$.
\end{lemma}

\begin{proof}
Put $\ualpha=(\alpha_0,\alpha_1,\alpha_2)$ and write
\[
 P=\sum_{j+k\le D} a_{j,k} X_0^j X_1^{D-j-k} X_2^k.
\]

If $\alpha_0=0$, we have, for $i=0,1,\dots,D$,
\[
 0 = \cD^iP(\ualpha)
   = \Big( X_2 \frac{\partial}{\partial X_2} \Big)^i P(\ualpha)
   = \sum_{k=0}^D a_{0,k} k^i \alpha_1^{D-k} \alpha_2^k.
\]
As the matrix $\big( k^i \big)_{0\le i\le D,\, 0\le k\le D}$ is invertible, this yields $a_{0,k} \alpha_1^{D-k} \alpha_2^k = 0$ for $k=0,\dots,D$.  However, as $X_0\nmid P$, we also have $a_{0,k}\neq 0$ for at least one of these values of $k$, and thus we conclude that $\alpha_1\alpha_2=0$.

Similarly, if $\alpha_2=0$, we find, for $i=0,1,\dots,D$,
\[
 0 = \cD^iP(\ualpha)
   = \Big( X_0 \frac{\partial}{\partial X_1} \Big)^i P(\ualpha)
   = \sum_{j=0}^{D-i} a_{j,0} \frac{(D-j)!}{(D-i-j)!} \alpha_0^{i+j}\alpha_1^{D-i-j}.
\]
As $X_2\nmid P$, we also note that $a_{j,0}\neq 0$ for some $j$ with $0\le j \le D$.  If $j_0$ is the smallest such index then, for $i=D-j_0$, this yields $0= a_{j_0,0} (D-j_0)! \alpha_0^D$ and so $\alpha_0=0$.

These two facts show that, if $\alpha_0\alpha_2=0$, then either $\alpha_0=\alpha_1=0$ or $\alpha_0=\alpha_2=0$, another way of formulating the lemma.
\end{proof}

\begin{remark} Conversely, if $\alpha=(0:1:0)$ (resp.~$\alpha=(0:0:1)$), then, for any integer $D\ge 1$, the point $\alpha$ is a common zero of the polynomials $P, \cD P, \dots,\cD^DP$ where $P=X_0^D+X_2^D$ (resp.~$P=X_0^D+X_1^D$) is not divisible by $X_0$ nor by $X_2$.
\end{remark}

\begin{proposition}
\label{constr:prop:estimates_deg_height}
Let $D, T \in \bN^*$, let $P\in\bC[\uX]_D$ with $X_0\nmid P$ and $X_2\nmid P$, and let $Y\in \bR$.  Suppose that
\[
 T \le \binom{D+1}{2},
 \quad
 \max\{25\log(3)D, \log\|P\|,\log\|\cD P\|,\dots,\log\|\cD^DP\|\} \le Y,
\]
and that $W=\cZ(\cD^iP\,;\, 0\le i < D+T)$ is not empty.  Then any irreducible component $Z$ of $W$ in $\bP_\bQ^2$ has dimension $0$ with
\begin{equation}
 \label{constr:prop:estimates_deg_height:eq}
 \deg(Z) \le \frac{D^2}{T} \et h(Z) \le \frac{3DY}{T}.
\end{equation}
\end{proposition}

\begin{proof}
By Lemma \ref{result:lemma:gcd}, the polynomials $P,\cD P,\dots,\cD^DP$ are relatively prime as a set.  Since they are all homogeneous of degree $D$, we conclude that there exist integers $a_1,\dots,a_D$ of absolute values at most $D$ such that $Q=\sum_{i=1}^D a_i\cD^i P$ is relatively prime to $P$.  Then $\cZ(P,Q)$ has dimension $0$ and since $W$ is a closed subset of $\cZ(P,Q)$, it also has dimension $0$.

Let $Z$ be an irreducible component of $W$.  Since $Z\subseteq \cZ(\cD^iP\,;\, 0\le i\le D)$, Lemma \ref{constr:lemma:ZerosOutsideG} shows that either $Z(\bC)$ is contained in the open set $\cG$ of $\bP^2(\bC)$ or it consists of one of the points $(0:1:0)$ or $(0:0:1)$ (the points of $Z(\bC)$ are conjugate over $\bQ$).  In the latter case, we have $\deg(Z)=1$ and $h(Z)=0$, and the estimates \eqref{constr:prop:estimates_deg_height:eq} follow.  Thus, in order to prove these estimates, we may assume, without loss of generality that $Z(\bC)\subseteq \cG$.

Let $G\colon\bC[\uX]_D\to\bC$ be a Chow form of $Z$ in degree $D$, and let $F\colon\bC[\uX]_D\to\bC$ denote the map given by $F(R)=\Res_D(P,Q,R)$ for each $R\in\bC[\uX]_D$, where $\Res_D$ denote the resultant in degree $D$.  We claim that $G^T$ divides $F$.

To prove this claim, choose a system of representatives $\ualpha_1,\dots,\ualpha_s\in\bC^3\setminus\{0\}$ of the points of $Z(\bC)$ and complete it to a system of representatives $\ualpha_1,\dots,\ualpha_t$ of those of $\cZ(P,Q)(\bC)$.  Then, there exist $a,b\in\Cmult$ and $e_1,\dots,e_t\in\bN^*$ such that
\begin{equation}
 \label{constr:prop:estimates_deg_height:eq1}
 F(R) = a R(\ualpha_1)^{e_1} \cdots R(\ualpha_t)^{e_t}
 \et
 G(R) = b R(\ualpha_1) \cdots R(\ualpha_s)
\end{equation}
for each $R\in\bC[\uX]_D$.  Moreover, $e_1=\dots=e_s$ represents the multiplicity of $G$ as an irreducible factor of $F$ over $\bQ$. So, our claim reduces to showing that $e_1\ge T$.  Denote by $\alpha$ the point of $Z(\bC)$ corresponding to $\ualpha_1$.  According to Proposition \ref{basic:prop:avoidingzeros}, there exists a polynomial $R$ in $I^{(\alpha,T)}_D$ such that $R(\ualpha_i)\neq 0$ for $i=2,\dots,t$.  Since $P$ and $Q$ also belong to $I^{(\alpha,T)}_D$, Corollary \ref{result:cor:1pt} shows that the resultant in degree $D$ vanishes to order at least $T$ at the point $(P,Q,R)$.  Therefore, for any fixed $S\in\bC[\uX]_D$, the polynomial $F(R+zS)\in\bC[z]$ is divisible by $z^T$.  Choosing $S$ so that $S(\ualpha_1)\neq 0$, the formula \eqref{constr:prop:estimates_deg_height:eq1} for $F$ provides
\[
 F(R+zS)
  = a S(\ualpha_1)^{e_1} R(\ualpha_2)^{e_2}\cdots R(\ualpha_t)^{e_t} z^{e_1} + \cO(z^{e_1+1}),
\]
and therefore $e_1\ge T$.

Since $G^T$ divides $F$, we obtain
\begin{equation}
 \label{constr:prop:estimates_deg_height:eq2}
 T\deg(Z) = T\deg(G) \le \deg(F)=D^2
\end{equation}
which proves the first half of \eqref{constr:prop:estimates_deg_height:eq}.  In terms of the convex body $\cB$ of Lemma \ref{prelim:lemma:heightB}, we also find, thanks to \cite[Prop.~3.7 (i) and Lemma 3.3 (i)]{LR},
\[
 T h_\cB(Z) = T h_\cB(G) \le h_\cB(F)+ 2D^2\log\binom{D+2}{2} \le h_\cB(F)+ 2\log(3)D^3.
\]
To translate this inequality in terms of the standard height $h(Z)$, we first observe that Lemma \ref{prelim:lemma:heightB} and the degree estimate \eqref{constr:prop:estimates_deg_height:eq2} lead to
\[
 DTh(Z) \le Th_\cB(Z) + 4\log(3)DT\deg(Z) \le Th_\cB(Z) + 4\log(3)D^3.
\]
Moreover, since $F$ is obtained by specializing the first two arguments of a Chow form of $\bP^2$ into $P$ and $Q$ with $\|P\|\le e^Y$ and $\|Q\|\le D^2e^Y \le e^{D+Y}$, and since that Chow form is homogeneous of degree $D^2$ in each of its three arguments, we also find
\[
 h_\cB(F) \le D^2Y + D^2(D+Y) + h_\cB(\bP^2) \le 2D^2Y + 19\log(3)D^3,
\]
using the upper bound for $h_\cB(\bP^2)$ provided by Lemma \ref{prelim:lemma:heightB}.  Combining the last three estimates, we conclude that
\[
 DTh(Z) \le 2D^2Y + 25\log(3)D^3 \le 3D^2Y,
\]
which proves the second half of \eqref{constr:prop:estimates_deg_height:eq}.
\end{proof}

\begin{remark} By Lemma 4.2 and Proposition 7.1 of \cite{LR}, the map $F$ is a Chow form of the intersection product $\div(P)\cdot\div(Q)$ and is therefore divisible by $G^\ell$ where $\ell$ is the intersection multiplicity of $\div(P)$ and $\div(Q)$ in $Z$.  From there, one can also prove that $\ell \ge T$ based on a standard algebraic definition of that intersection multiplicity.
\end{remark}

%
%

\section{Proof of Theorem \ref{intro:thm:main}}
\label{sec:proof}

Let the notation and hypotheses be as in Theorem \ref{intro:thm:main}.  We also put $\gamma=(\xi,\eta)$ and use the notation of Section \ref{sec:distance}.  We shall argue by contradiction, assuming on the contrary that $(1:\gamma)$ is not a point of $\bP^2(\Qbar)$.    From there we proceed in several steps.

\subsection*{Step 1.} For each positive integer $D$, we define a convex body $\cC_D$ of $\bC[\uX]_D$ by
\[
 \cC_D
  = \Big\{ P\in\bC[\uX]_D \,;\,
       \|P\| \le \exp(2D^\beta)
       ,\,
       \max_{0\le i< \lfloor D^\tau \rfloor} |\cD^iP(\utheta)| \le \exp(-(1/2)D^\nu)
    \Big\},
\]
and denote by $\tP_D$ the homogeneous polynomial of $\bZ[\uX]_D$ determined by the condition
\[
 \tP_D(1,X_1,X_2) = X_1^a X_2^{-b} P_D(X_1,X_2)
\]
where $b$ stands for the largest integer such that $X_2^b$ divides $P_D$, and where $a=D-\deg(P_D)+b$.  Then, by construction, $\tP_D$ is not divisible by $X_0$ nor by $X_2$.  We claim that, for any sufficiently large $D$, the polynomials $\cD^j\tP_D$ with $0\le j < 2\lfloor D^\tau \rfloor$ all belong to $\cC_D$.\\

To prove this, fix a choice of integer $j$ with $0\le j< 2\lfloor D^\tau \rfloor$, and put $Q=\cD^j\tP_D$. Using Lemma \ref{basic:lemma:derivatives}, we find
\[
 \|Q\| \le D^j \cL(\tP_D) \le D^j (D+1)^2 \|P_D\| \le (D+1)^{2D^\tau+1} \exp(D^\beta) =\exp((1+o(1))D^\beta).
\]
Moreover, for any $i=0,\dots,\lfloor D^\tau \rfloor-1$, Leibniz' rule of differentiation for a product leads to
\[
 \begin{aligned}
   |\cD^i Q(1,\gamma)|
    &=|\cD_1^{i+j} (X_1^a X_2^{-b} P_D(X_1,X_2))|_{X_1=\xi,\,X_2=\eta} \\
    &\le \sum_{r+s+t=i+j} \frac{(i+j)!}{r!\,s!\,t!}\, |\cD_1^r X_1^a|_{X_1=\xi}\,
                          |\cD_1^s X_2^{-b}|_{X_2=\eta}\, |\cD_1^t P_D(\xi,\eta)| \\
    &\le \sum_{r+s+t=i+j} \frac{(i+j)!}{r!\,s!\,t!} a^r \max\{1,|\xi|\}^a\, b^s\, |\eta|^{-b}
         \max_{0\le k < 3\lfloor D^\tau \rfloor} |\cD_1^k P_D(\xi,\eta)| \\
    &\le \max\{1,|\xi|,|\eta|^{-1}\}^{a+b} (a+b+1)^{3\lfloor D^\tau \rfloor} \exp(-D^\nu) \\
    &= \exp(-(1-o(1))D^\nu).
 \end{aligned}
\]

\subsection*{Step 2.}  Define $\delta = \nu +\tau-2-\beta$, and fix an arbitrarily large integer $D$.  Then, put
\[
 T=\lfloor D^\tau \rfloor,  \quad  Y=2D^\beta,  \quad  U=D^\nu/2,
\]
and define a real number $C$ by the condition $TU=CD^2Y$.  Then, the convex set $\cC$ defined in Proposition \ref{constr:prop:convex} coincide with $\cC_D$ and, assuming that $D$ is sufficiently large, the hypotheses of Proposition \ref{constr:prop:subvar} are all fulfilled (because $1\le \tau < \min\{2,\beta\}$ and $\tau+\nu > 2+\beta$), and moreover we have $C''=(C-5)/6\ge D^\delta/25$.  Thus, there exists a $0$-dimensional subvariety $Z=Z_D$ of $\bP^2_\bQ$ contained in $\cZ(\cD^i\tP_D\,;\,0\le i< 2T)$ such that
\[
 h_{\cC_D}(Z) \le - \frac{D^\delta}{25} ( 2D^\beta\deg(Z) + Dh(Z) ).
\]
Let $\uZ$ be a set of representatives of the points of $Z(\bC)$ by elements of $\bC^3$ of norm $1$.  By Proposition \ref{prelim:prop:dim0}, we have
\[
 \sum_{\ualpha\in\uZ} \log \sup\{|P(\ualpha)| \,;\, P\in\cC_D\}
 \le
 h_{\cC_D}(Z) - Dh(Z) + 9\log(3)D\deg(Z).
\]
For each $\alpha\in Z(\bC)$ with corresponding point $\ualpha\in\uZ$, we also have, according to the definitions,
\[
 \sup\{|P(\ualpha)| \,;\, P\in\cC_D\}
  \ge
  \sup\{ |P(\ualpha)| \,;\, P\in I_D^{(\gamma,T)}, \|P\| \le 1\ \}
  = |I_D^{(\gamma,T)}|_\alpha.
\]
Now, let $\cU$ denote the set of points $\alpha$ of $Z(\bC)$ with $\dist(\alpha,(1:\gamma))\le (2c_2)^{-1}$.  For each $\alpha \in Z(\bC) \setminus \cU$, Proposition \ref{distance:prop:min|I|} gives
\[
 |I_D^{(\gamma,T)}|_\alpha
   \ge c_5^{-T} T^{-6T\log(T)} \dist(\alpha,(1:\gamma))^T
   \ge (2c_2c_5)^{-T} T^{-6T\log(T)}
   \ge T^{-7T\log(T)},
\]
assuming $D$ large enough so that $T\ge 2c_2c_5$.  For the more interesting points $\alpha\in\cU$, it gives
\[
 \begin{aligned}
 |I_D^{(\gamma,T)}|_\alpha
   &\ge c_4^{-1} c_5^{-T} T^{-6T\log(T)} \max\{\dist(\alpha,(1:\gamma))^T,\dist(\alpha,A_\gamma)\}\\
   &\ge T^{-7T\log(T)} \max\{\dist(\alpha,(1:\gamma))^T,\dist(\alpha,A_\gamma)\},
 \end{aligned}
\]
provided that $D$ is large enough.  Putting all these estimates together, taking into account that $Z(\bC)$ consists of $\deg(Z)$ points, we conclude that
\[
 \begin{aligned}
 \sum_{\alpha\in\cU} \max\{ T\log\dist(\alpha,(1:\gamma)), \log\dist(\alpha,A_\gamma) \}
 &\le
 \sum_{\alpha\in Z} \log |I_D^{(\gamma,T)}|_\alpha + 7 T (\log T)^2 \deg(Z) \\
 &\le - \frac{D^\delta}{25} ( D^\beta\deg(Z) + Dh(Z) )
 \end{aligned}
\]
if $D$ is large enough (because $\beta > \tau \ge 1$).  In particular, the set $\cU$ is not empty and contains at least one point $\alpha_0$ for which $\log \dist(\alpha_0,(1:\gamma)) \le -D^{\delta+\beta}/(25T)$.   Thus, as $D$ goes to infinity, the point $\alpha_0$ runs through an infinite sequence of points of $\bP^2(\Qbar)$ converging to $(1:\gamma)$ but distinct from $(1:\gamma)$ (because $(1:\gamma)\notin\bP^2(\Qbar)$).

\subsection*{Step 3.}  Denote by $D^*$ the smallest positive integer for which
\[
 Z \subseteq \cZ\big( \cD^i \tP_{D^*+1}\,;\, 0\le i < 2\lfloor(D^*+1)^\tau\rfloor \big).
\]
If $D\ge 2$, such an integer exists and is at most equal to $D-1$. Moreover, $D^*$ goes to infinity with $D$  because, as $\tP_{D^*+1}$ is not divisible by $X_0$ nor by $X_2$, it follows from Lemma \ref{result:lemma:gcd} that $\cZ(\cD^i \tP_{D^*+1}\,;\, 0\le i\le D^*+1)(\bC)$ is a finite subset of $\bP^2(\Qbar)$ and so, for fixed $D^*\ge 1$, this set does not contain the point $\alpha_0$ of $Z(\bC)$ when $D$ is large enough. Thus, assuming $D$ large enough, it follows from Step 1 that $\cD^i\tP_{D^*+1}$ belongs to $\cC_{D^*+1}$ for $i=0,\dots,D^*+1$ and consequently
\[
 \max\{25\log(3)(D^*+1),\,\log\|\tP_{D^*+1}\|,\dots,\log\|\cD^{D^*+1}\tP_{D^*+1}\|\}
 \le
 2(D^*+1)^\beta.
\]
For $D$ large enough, we also have $\lfloor(D^*+1)^\tau\rfloor \le \binom{D^*+2}{2}$ and, by Proposition \ref{constr:prop:estimates_deg_height}, we conclude that
\[
 \deg(Z) \le \frac{(D^*+1)^2}{\lfloor(D^*+1)^\tau\rfloor}\le 2(D^*)^{2-\tau}
 \et
 h(Z) \le \frac{6(D^*+1)^{1+\beta}}{\lfloor(D^*+1)^\tau\rfloor}\le 7(D^*)^{1+\beta-\tau}.
\]

\subsection*{Step 4.}  For $D$ large enough, we have $D^*\ge 2$ and so, by the very choice of $D^*$, there exists an integer $i_0$ with $0\le i_0 < 2\lfloor(D^*)^\tau\rfloor$ such that $Z$ is not contained in the curve of $\bP^2$ defined by the polynomial $P^*:=\cD^{i_0} \tP_{D^*}$.  For $D$ large enough, we also have $P^*\in\cC_{D^*}\cap\bZ[\uX]$.  Then, Proposition \ref{prelim:prop:dim0} gives
\[
 0 \le 7\log(3)D^*\deg(Z) + D^*h(Z) + \sum_{\ualpha\in\uZ}\log|P^*(\ualpha)|.
\]
Moreover, the fact that $P^*\in\cC_{D^*}$ leads to the crude estimate
\[
 \max_{\ualpha\in\uZ}\log|P^*(\ualpha)|
  \le D^*\log(3) + \log \|P^*\|
  \le 4(D^*)^\beta.
\]
Combining the last two results , we deduce that, for $D$ large enough,
\begin{equation}
 \label{step4:eq1}
 \sum_{\ualpha\in\uZ}\min\{0,\log|P^*(\ualpha)|\}
  \ge - 5(D^*)^\beta \deg(Z) - D^*h(Z).
\end{equation}

Put $T^*=\lfloor(D^*)^\tau\rfloor$.  Then, for a point $\alpha\in\cU$ with representative $\ualpha\in\uZ$, Proposition \ref{distance:prop:devP} provides the more precise estimate
\[
 \begin{aligned}
 |P^*(\ualpha)|
  &\le
  c_4\max_{0\le i <T^*} |\cD^iP^*(1, \gamma)|
  + c_4^{D^*}\|P^*\| \big(\dist(\alpha,(1:\gamma))^{T^*}+\dist(\alpha,A_\gamma)\big)\\
  &\le
  c_4e^{-(D^*)^\nu/2}
  + c_4^{D^*}e^{2(D^*)^\beta} \big(\dist(\alpha,(1:\gamma))^{T^*}+\dist(\alpha,A_\gamma)\big).
 \end{aligned}
\]
However, if $D^*$ is large enough, the inequality \eqref{step4:eq1} combined with the estimates for $\deg(Z)$ and $h(Z)$ obtained in Step 3 leads to
\[
 \log |P^*(\ualpha)|
   \ge - 5(D^*)^\beta \deg(Z) - D^*h(Z)
   \ge - 17(D^*)^{2+\beta-\tau},
\]
thus $|P^*(\ualpha)| \ge 2c_4e^{-(D^*)^\nu/2}$, and so
\[
 \log |P^*(\ualpha)|
  \le
  3(D^*)^\beta +\max\big\{ T^*\log\dist(\alpha,(1:\gamma)),\, \log\dist(\alpha,A_\gamma)\big\}.
\]
Note that this holds for any $\alpha\in\cU$ with a lower bound on $D^*$ not depending on $\alpha$.  Therefore, if $D$ is large enough, we conclude using \eqref{step4:eq1} that, for any subset $\cS$ of $\cU$, we have
\begin{equation*}
 \sum_{\alpha\in\cS} \max\big\{ T^* \log\dist(\alpha,(1:\gamma)),\, \log\dist(\alpha,A_\gamma)\big\}
  \ge - 8(D^*)^\beta \deg(Z) - D^*h(Z).
\end{equation*}

\subsection*{Step 5.}
According to the last estimate from Step 2, we have
\[
 T \sum_{\alpha\in\cU'} \log\dist(\alpha,(1:\gamma))
 + \sum_{\alpha\in\cU''} \log\dist(\alpha,A_\gamma)
  \le - \frac{D^\delta}{25} ( D^\beta\deg(Z) + Dh(Z) )
\]
for any partition of $\cU$ into disjoint subsets $\cU'$ and $\cU''$.  We choose
\[
 \cU' = \{ \alpha\in\cU \,;\, T^* \log\dist(\alpha,(1:\gamma)) \ge \log\dist(\alpha,A_\gamma)\}
 \et
 \cU''= \cU\setminus\cU'.
\]
Then, the last estimate of Step 4 applied to the sets $\cU'$ and $\cU''$ gives respectively
\[
 \begin{aligned}
 T^* \sum_{\alpha\in\cU'} \log\dist(\alpha,(1:\gamma))
  &\ge - 8(D^*)^\beta \deg(Z) - D^*h(Z),\\
 \sum_{\alpha\in\cU''} \log\dist(\alpha,A_\gamma)
  &\ge - 8(D^*)^\beta \deg(Z) - D^*h(Z).
 \end{aligned}
\]
Combining these three inequalities, we obtain
\[
 - \frac{D^\delta}{25} ( D^\beta\deg(Z) + Dh(Z) )
 \ge - \Big( \frac{T}{T^*} + 1 \Big) \big( 8(D^*)^\beta \deg(Z) + D^*h(Z) \big)
\]
and so
\[
 D^{\delta+\beta}\deg(Z) + D^{\delta+1}h(Z) \ll D^\tau (D^*)^{1-\tau} h(Z)
\]
(we may omit the term $D^\tau (D^*)^{\beta-\tau} \deg(Z)$ in the right hand side as it is negligible with respect to $D^{\delta+\beta}\deg(Z)$).  This means that
\begin{equation}
 \label{step5:eq1}
 D^{\delta+\beta-\tau}\deg(Z) \ll (D^*)^{1-\tau} h(Z)
 \et
 (D^*)^{\tau-1} \ll D^{\tau-\delta-1}.
\end{equation}
Since $\deg(Z)\ge 1$ and $h(Z)\ll (D^*)^{1+\beta-\tau}$ (see Step 3), the first estimate in \eqref{step5:eq1} implies that
\[
 D^{\delta+\beta-\tau} \ll (D^*)^{\beta +2-2\tau}.
\]
As $\tau\ge 1$, combining this with the second estimate from \eqref{step5:eq1} yields
\[
 (\tau-1)(\delta+\beta-\tau) \le (\tau-\delta-1)(\beta +2-2\tau)
\]
which after simplifications is equivalent to $\delta \le (\tau-1)(2-\tau)/(\beta+1-\tau)$.  This contradicts  the hypothesis on $\nu$ in \eqref{intro:thm:main:eq1}, and therefore proves that $\xi,\eta\in\Qbar$.  The remaining assertion of Theorem \ref{intro:thm:main} follows from this fact, as explained in the introduction.

%
%


\begin{thebibliography}{99}


\baselineskip=13pt

\bibitem{BM}
  W.~D.~Brownawell and D.~W.~Masser,
  \textit{Multiplicity estimates for analytic functions II},
  Duke Math.\ J.\ \textbf{47} (1980), 273--295.

\bibitem{Fi}
  S.~Fischler, \textit{Interpolation on algebraic groups},
  Compos.\ Math.\ \textbf{141}  (2005), 907--925.

\bibitem{LR}
  M.~Laurent and D.~Roy,
  \textit{Criteria of algebraic independence with multiplicities and approximation by hypersurfaces},
  J.\ reine angew.\ Math.\ \textbf{536} (2001), 65--114.

\bibitem{Mah} K.~Mahler,
  \textit{On a class of entire functions},
  Acta Math.\ Acad.\ Sci.\ Hungar.\ \textbf{18} (1967), 83-–96.

\bibitem{Ma} D.~W.~Masser,
  \textit{On polynomials and exponential polynomials in several complex variables},
  Invent.\ Math.\ \textbf{63} (1981), 81--95.

\bibitem{Ne1} Yu.~V.~Nesterenko,
  \textit{Estimates for the orders of zeros of functions of a certain class and their applications in the theory of transcendental numbers}, Izv.\ Akad.\ Nauk SSSR Ser.\ Mat.\ \textbf{41} (1977), 253-284;
  English transl.\ in Math.\ USSR Izv.\ \textbf{11} (1977), 239-270.

\bibitem{Ne2} Yu.~V.~Nesterenko,
  \textit{On algebraic independence of algebraic powers of algebraic numbers},
  Mat.\ Sb.\ \textbf{123} (1984), 435--459;
  English transl.\ in Math.\ USSR Sb.\ \textbf{51} (1985), 429--454.

\bibitem{Ph}
  P.~Philippon,
  \textit{Crit\`eres pour l'ind\'ependance alg\'ebrique},
  Pub.\ Math.\ IHES {\bf 64} (1986), 5--52.

\bibitem{Ph1994}
  P.~Philippon, \textit{Quatre expos\'es sur la th\'eorie de l'\'elimination},
  in: Quelques aspects de la th\'eorie de l'\'elimination,
  pr\'etirage no.~94-25, Laboratoire de math\'ematiques discr\`etes,
  CNRS, 1994, 2--46; http://www.math.jussieu.fr/$\sim$pph/DEACIRM93.pdf

\bibitem{Re}
  G.~R\'emond,
  \textit{\'Elimination multiprojective} (Chapter 5) and \textit{G\'eom\'etrie diophantienne
  multiprojective} (Chapter 7), in: Introduction to algebraic independence theory,
  Eds: P.~Philippon and Yu.~Nesterenko, Lecture Notes in Math.~1752, Springer-Verlag, 2001.

\bibitem{R2001}
  D.~Roy,
  \textit{An arithmetic criterion for the values of the exponential function},
  Acta Arith.\ \textbf{97} (2001), 183--194.

\bibitem{R2002}
  D.~Roy,
  \textit{Interpolation formulas and auxiliary functions},
  J.~Number Theory \textbf{94} (2002), 248--285.

\bibitem{R2008}
  D.~Roy,
  \textit{Small value estimates for the multiplicative group},
  Acta Arith.\ \textbf{135} (2008), 357-393.

\bibitem{R2010}
  D.~Roy,
  \textit{Small value estimates for the additive group},
  Int.~J.~Number Theory \textbf{6} (2010), 919-956.

\bibitem{Ti}
  R.~Tijdeman,
  \textit{An auxiliary result in the theory of transcendental numbers},
  J.~Number Theory \textbf{5} (1973), 80--94.

\end{thebibliography}
\end{document}